\documentclass[reqno]{amsart}
\usepackage{amsfonts,latexsym,amsmath, amssymb}
\usepackage{mathrsfs,enumerate}
\usepackage{bm}
\usepackage{color}
\usepackage{amsmath}
\usepackage{cite}
\usepackage{graphicx,enumerate}

\usepackage{enumitem}

\newcommand{\bea}{\begin{eqnarray}}
\newcommand{\eea}{\end{eqnarray}}
\def\beaa{\begin{eqnarray*}}
\def\eeaa{\end{eqnarray*}}
\def\ba{\begin{array}}
\def\ea{\end{array}}
\def\be#1{\begin{equation} \label{#1}}
\def \eeq{\end{equation}}

\def\be{{\beta}}

\def\AA{{\mathcal A}}

\def\FF{{\mathcal F}}

\def\LL{{\mathcal L}}
\def\GG{{\mathcal G}}

\def\AA{{\mathcal A}}

\def\R{{\mathbb{R}}}
\def\C{{\mathbb{C}}}

\def\N{{\mathbb N}}

\def\sgH2{\sigma_H^2}
\def\sgL2{\sigma_L^2}

\newtheorem{theorem}{Theorem}[section]
\newtheorem{lemma}[theorem]{Lemma}

\newtheorem{corollary}[theorem]{Corollary}
\newtheorem{definition}[theorem]{Definition}
\newtheorem{remark}[theorem]{Remark}

\numberwithin{equation}{section}

\allowdisplaybreaks
\begin{document}

\title[Stability in a credit rating migration problem]{Stability of traveling wave solutions in a credit rating migration Free Boundary Problem}
\author[C.-M. Brauner]{Claude-Michel Brauner}
\address{Institut de Math\unexpanded{\'e}matiques de Bordeaux, Universit\unexpanded{\'e} de Bordeaux, 33405 Talence, France}
\email{claude-michel.brauner@u-bordeaux.fr}
\author[Y. Dong]{Yuchao Dong}
\address{School of Mathematical Sciences, Tongji University, Shanghai 200092, China}
\email{ycdong@tongji.edu.cn}
\author[J. Liang]{Jin Liang}
\address{School of Mathematical Sciences, Tongji University, Shanghai 200092, China}
\email{liang\_jin@tongji.edu.cn}
\author[L. Lorenzi]{Luca Lorenzi}
\address{Plesso di Matematica, Dipartimento di Scienze Matematiche, Fisiche e Informatiche, Universit\`a di Parma, Parco Area delle Scienze 53/A, I-43124 Parma, Italy }
\email{luca.lorenzi@unipr.it}


\begin{abstract}
In this paper, we study the stability of traveling wave solutions arising from a credit rating migration problem with a free boundary, After some transformations, we turn the Free Boundary Problem into a fully nonlinear parabolic problem on a fixed domain and establish a rigorous stability analysis of the equilibrium in an exponentially weighted function space. It implies the convergence of the discounted value of bonds that stands as an attenuated traveling wave solution. 
\end{abstract}
\maketitle
\section{Introduction}

\subsection{Financial motivation}
The evolution and globalization of the financial market have amplified the significance of credit risks, often elevating them to pivotal roles within the market dynamics. This influence became unmistakably evident during the 2008 financial crisis and the more recent wave of defaults among small banks in the United States. These instances serve as poignant illustrations of the profound impact that credit risks can exert on the financial landscape. Consequently, the imperative of effectively managing credit risk has surged in importance over time, underlining the critical need for robust risk mitigation strategies and practices.

Credit risks encompass two primary facets: default risk and credit rating migration risk. Default risk signifies the likelihood that a company will fail to meet its debt obligations, while credit rating migration risk pertains to alterations in a company's credit rating triggered by shifts in its financial condition, typically impacting the company's equities and the potential risks associated with related financial contracts. The quantification of credit risk through mathematical methodologies falls within two distinct frameworks: structural models and intensity models. The intensity model posits that credit risk emanates from external factors, often modeled using Markov chains, as referenced in \cite{Ma}. In this context, the occurrences of default and/or credit rating migrations hinge on an external transition intensity, as elucidated by Jarrow and Longstaff in \cite{JLT,JT}. However, in practical markets, it is evident that a company's prevailing financial condition significantly influences the probabilities of default and credit rating migrations. To incorporate this endogenous factor, structural models emerge as pivotal tools for credit risk modeling, with historical roots tracing back to Merton's seminal work in 1974, as outlined in \cite{M}. Within such models, credit rating migrations and defaults are linked to the firm's asset value and its outstanding obligations.

In 2015, Liang and Zeng \cite{LZ} first study the pricing problem of the corporate bond with credit rating migration risk, where a predetermined migration threshold is given to divide asset value into high and low rating regions, in which the asset value follows different stochastic processes. Hu, Liang, and Wu \cite{HLW} further develop this model, where the migration boundary is a free boundary governed by a ratio of the firm's asset value and debt. Furthermore, one contribution of their work is to show that this problem admits an asymptotic traveling wave. Dong, Liang and Brauner \cite{DLB} further extend this framework to include default in the model. They also find the asymptotic traveling wave. For some other works relative to this model, the readers can see the survey paper \cite{CHLY}

In previous research, the asymptotic behavior takes an important position and it also brings up stability problem. Besides its own mathematical interest, the study of stability is also of financial importance. As the traveling wave has a closed form, stability means the approximation of the original solution by the traveling wave. In \cite{HLW}, the asymptotic traveling wave equation offers insights into pricing a specific debt instrument subject to credit rating risk. This pricing methodology hinges on evaluating the discounted value of a company's assets. Given that \cite{HLW} focuses on a particular type of debt, a natural question is whether this pricing rule is suitable for other kinds of long-term debts and robust when confronted with slight variations in model coefficients. From a mathematical point of view, this question is closely related  to the stability of the traveling wave solutions. If stability can be established, it signifies the existence of a robust and general pricing framework with credit risk, thereby expanding its utility beyond the specific debt instrument initially considered in the study.
\subsection{The model}
Throughout the paper we will adopt the following notations as far as financial parameters are concerned:
\begin{itemize}
\item $\delta>0$ is the risk discount rate;
\item $0<\gamma<1$ is the threshold proportion of the debt and value of the firm's rating;
\item $r>0$ is the risk free interest rate;
\item $\sigma_H>0$ and $\sigma_L>0$ represent the volatilities of the firm under the high and low credit grades, respectively.
\end{itemize}
They verify the condition (see \cite[(4.12)]{LWH16}),
\begin{eqnarray}\label{admissible}
\frac{1}{2}\sgH2 <\delta< \frac{1}{2}\sgL2,
\end{eqnarray}
which hereafter defines the domain $(\AA)$ of financially admissible parameters.

Our starting point is the Free Boundary Problem for the discounted value of the bond in finance $\phi$ (see \cite[Section 3]{LWH16})
\begin{align}\label{equ_phi}
\frac{\partial \phi}{\partial t}(t,x) -\frac{1}{2}\sigma^2\frac{\partial^2 \phi}{\partial x^2}(t,x) -\left(r-\frac{1}{2}\sigma^2\right)\frac{\partial \phi}{\partial x}(t,x) +r\phi(t,x)=0,
\end{align}
where $x\in\R$ is the spatial coordinate and $t$ denotes the time. In the equation \eqref{equ_phi},
\begin{align*}
\sigma= \sigma(\phi,x)= \sigma_H +(\sigma_L-\sigma_H)H(\phi-\gamma e^{x-\delta t}),
\end{align*}
where $H$ is the Heaviside function.

The free boundary (or free interface) $s(t)$ and the interface conditions are
 \begin{equation}
\left\{
\begin{aligned}\label{interface_phi}
 &\big[\phi(t,\cdot)\big]_{x=s(t)} = \bigg [\frac{\partial \phi}{\partial x}(t,\cdot)\bigg ]_{x=s(t)} =0, \\
&\phi(t,s(t))= \gamma e^{s(t)-\delta t},
 \end{aligned}
 \right.
 \end{equation}
 where $\big[f\big]_{\xi=\xi_0} =f(\xi_0^+)-f(\xi_0^-)$.

In order to establish convergence, as $t$ tends to $+\infty$, of $\phi(t,x)$ to a traveling wave solution, it is convenient to make the following change of dependent variable (see \cite[Section 3]{LWH16})
\begin{align}
u(t,x)= e^{rt}\phi(t,x).
\label{u-phi}
\end{align}
From \eqref{equ_phi} and \eqref{interface_phi} it follows that $u$ satisfies the differential equation
  \begin{align}\label{equ_u}
  \frac{\partial u}{\partial t}(t,x) -\frac{1}{2}\sigma^2\frac{\partial^2 u}{\partial x^2}(t,x) -\left(r-\frac{1}{2}\sigma^2\right)\frac{\partial u}{\partial x}(t,x)=0,
  \end{align}
   with interface conditions at $x=s(t)$:
   \begin{equation}
\left\{
\begin{aligned}\label{interface_u}
&\big[u(t)\big]_{x=s(t)} = \bigg [\frac{\partial u}{\partial x}(t)\bigg ]_{x=s(t)} =0, \\
&u(t,s(t))= \gamma e^{s(t)+(r-\delta)t},
\end{aligned}
\right.
\end{equation}
  In addition, the following boundary conditions hold at $\pm\infty$:
\begin{eqnarray}
\lim_{x \to -\infty } u(t,x) = 0, \quad  \lim_{x \to +\infty } u(t,x) = 1. \label{BC}
\end{eqnarray}

 We look for a traveling wave solution $K(x+ct)$ of system \eqref{equ_u}--\eqref{BC} which propagates at velocity $c=r-\delta$. In the moving coordinate (see \cite[Section 3]{LWH16})
   \begin{align*}
  \xi=x+ct,\qquad\;\, \eta(t)= s(t)+ct,
  \end{align*}
 the system for $u=u(t,\xi)$ and free boundary $\eta=\eta(t)$ reads
 \begin{equation}\label{S1}
 \begin{aligned}
 &\frac{\partial u}{\partial t}(t,\xi)  = \frac{1}{2}\sgL2 \frac{\partial^2 u}{\partial \xi^2}(t,\xi)  + \left(\delta -
 \frac{1}{2}\sgL2\right)\frac{\partial u}{\partial \xi}(t,\xi) , \quad \xi < \eta(t), \\
 &\frac{\partial u}{\partial t}(t,\xi) = \frac{1}{2}\sgH2 \frac{\partial^2 u}{\partial \xi^2}(t,\xi)  + \left(\delta
 - \frac{1}{2}\sgH2\right)\frac{\partial u}{\partial \xi}(t,\xi) , \quad \xi > \eta(t).
 \end{aligned}
 \end{equation}
 From \eqref{interface_u} we derive the following free interface conditions at $\xi=\eta(t)$:
 \begin{eqnarray}
\left\{
\begin{aligned}
&\big[u(t,\cdot)\big]_{\xi=\eta(t)} = \bigg [\frac{\partial u}{\partial \xi}(t,\cdot)\bigg ]_{\xi=\eta(t)} =0,\\
& u(t,\eta(t))= \gamma e^{\eta(t)},
\label{S4}
\end{aligned}
\right.
\label{S3-S4}
\end{eqnarray}
together with the boundary conditions:
 \begin{eqnarray}
 \lim_{\xi \to -\infty } u(t,\xi) = 0, \quad  \lim_{\xi \to +\infty } u(t,\xi) = 1. \label{S5}
 \end{eqnarray}

System \eqref{S1}-\eqref{S5} admits a unique steady solution $(K,\eta^*)$ which verifies the following system (see \cite[Lemma 4.7]{LWH16})
\begin{eqnarray*}
\left\{
\begin{array}{ll}
\displaystyle\frac{1}{2}\sgL2 \frac{d^2K}{d\xi^2} + \bigg (\delta - \frac{1}{2}\sgL2\bigg )\frac{dK}{d\xi} =0, \quad & \xi < \eta^*, \\[3mm]
\displaystyle\frac{1}{2}\sgH2 \frac{d^2K}{d\xi^2} + \bigg (\delta - \frac{1}{2}\sgH2\bigg )\frac{dK}{d\xi} =0, \quad & \xi > \eta^*,\\[3mm]
\displaystyle [K]_{\xi=\eta^*} = \bigg[\frac{dK}{d\xi}\bigg]_{\xi=\eta^*}=0, \quad K(\eta^*)= \gamma e^{\eta^*}, \\[3mm]
K(-\infty)=0, \quad K(+\infty)=1,
\end{array}
\right.
\end{eqnarray*}
which is defined by
\begin{equation*}
\eta^*=\ln\bigg (\frac{\sigma^2_L(2\delta-\sigma_H^2)}{2\delta\gamma(\sigma_L^2-\sigma_H^2)}\bigg )
\end{equation*}
and
\begin{eqnarray}
K(\xi)=\left\{
\begin{array}{ll}
\displaystyle\gamma e^{\eta^*}\exp\bigg [\bigg (1-\frac{2\delta}{\sigma_L^2}\bigg )(\xi-\eta^*)\bigg ], & \xi\le\eta^*,\\[4mm]
\displaystyle 1-(1-\gamma e^{\eta^*})\exp\bigg [\bigg (1-\frac{2\delta}{\sigma_H^2}\bigg )(\xi-\eta^*)\bigg ],\quad & \xi>\eta^*
\end{array}
\right.
\label{funct-K}
\end{eqnarray}
The convergence of $u(t,\xi)$ to $K(\xi)$ as $t$ tends to $+\infty$ was proved in \cite{LWH16} in the case of a particular initial condition. In this paper, we  prove the nonlinear stability of the traveling wave $K$ for all the admissible values of parameters, using the techniques of  \cite{ABLZ20,BL18,BHL00,BLZ20}
developed for applications in physics, especially in combustion theory. As we have already said, here stability indicates a robust general rule of pricing under credit risk and our mathematical results
confirm the validity of the model of \cite{LWH16}. In other scientific fields, domains of instability are more likely to be identified (see, e.g., \cite{BL21}).

The paper is organized as follows: in Section \ref{linearized_model}, we fix the free interface $\eta(t)$ to $\eta^*$. The new spatial variable is denoted by $y=\xi -\eta(t) + \eta^*$. Then, we consider the perturbations
\begin{eqnarray*}
v(t,y)=u(t,y)-K(y), \qquad\;\, f(t)=\eta(t)-\eta^*.
\end{eqnarray*}
According to the method of \cite{BHL00}, we introduce the ansatz:
\begin{eqnarray*}
v(t,y) =f(t)\frac{dK}{dy}(y) + w(t,y) \end{eqnarray*}
where $w$ is the new unknown. After eliminating $f$ and its derivative, we show that $w$ verifies the problem
\begin{equation}\label{FNLP_intro}
\left\{
\begin{array}{ll}
\displaystyle \frac{d w}{d t}(t)= \LL w(t) + \FF(w(t)),\\[3mm]
B_1 w(t)= \GG(w(t)).
\end{array}
\right.
\end{equation}
The system for $w$ does not contain the free interface which has been eliminated. It is a \textit{fully nonlinear parabolic problem} (see \cite{Lunardi96}) because the nonlinear term contains traces of second-order derivatives with respect to the spatial variable $y$ and, therefore, is of the same order as the linear operator.

The analysis of the stability of the traveling wave solution $K$ is now equivalent to the stability of the trivial solution $w=0$ of problem \eqref{FNLP_intro}.

In Section \ref{analysis}, we prove that the realization $L$ of the operator ${\mathcal L}$ in a suitable weighted space ${\mathscr{X}}$ of continuous functions generates an analytic semigroup. As it is well known,
the introduction of exponentially weighted spaces for proving stability of traveling waves has been widely used since the pioneering work of Sattinger (see \cite{S76}). The spectrum $\sigma(L)$  consists of the union of a halfline, strictly included in $(-\infty, 0)$, and the solutions of the dispersion relation ${\mathcal D}_{\lambda}=0$. Section \ref{dispersion_relation} is devoted to a thorough analysis of the dispersion relation. It turns out that, in the domain of financially admissible parameters, the dispersion relation has no solutions. Therefore, we establish the nonlinear stability of $w=0$ in Section \ref{nonlinear_stability}. The main result is Theorem \ref{stability}.

In the last Section \ref{stability_phi}, we return to the problem for the discounted value of the bond $\phi=e^{-rt} u$. We observe that, when $t\to +\infty$, $\phi(t,x)$ is asymptotically equivalent to the \textit{discounted or attenuated} traveling wave $\Phi(t,x)=e^{-rt} K(x+ct)$. It is worth noting that attenuation of traveling waves is a common phenomenon in physics due to loss of energy; for example, propagating seismic-waves attenuate with time because Earth is not perfectly elastic (see \cite{AVLW20}).

\medskip

\paragraph{\bf{Notation.}} Throughout the paper, the derivatives of functions which depend on a single variable will be denoted using a prime. For example, $K''$ will denote the second-order derivative of the function $K$.
Sometimes, we will find it convenient to write $D_y$ and $D_{yy}$ instead of $\frac{\partial}{\partial y}$ and $\frac{\partial^2}{\partial y^2}$, respectively.
Moreover, for every interval $I\subset\R$ and every $k\in\N\cup\{0\}$, we denote by $C^k_b(I;\C)$ the set of all functions $u:I\to\C$ which are continuously differentiable up to the order $k$ in $I$ and are bounded together with all their derivatives. This space is endowed with its natural norm $\|u\|_{C^k_b(I;\C)}=\sum_{j=0}^k\|u^{(j)}\|_{\infty}$ where $u^{(0)}=u$, and $\|\cdot\|_{\infty}$ denotes the sup-norm over $I$.

\section{The linearized model}
\label{linearized_model}

\subsection{The system with fixed interface}
We set $t'=t$,  $y=\xi -\eta(t) + \eta^*$, i.e. $y=x-s(t) + \eta^*$. When $\xi=\eta(t)$,  $y=\eta^*$, that is, the free interface is now at $y=\eta^*$.

In the coordinates $(t',y)$, omitting the prime for simplicity,  system \eqref{S1}-\eqref{S5} reads for $y \neq \eta^*$:
\begin{eqnarray}
\frac{\partial u}{\partial t} - \dot{\eta}\frac{\partial u}{\partial y} = \frac{1}{2}\sgL2 \frac{\partial^2 u}{\partial y^2} + \left(\delta -
\frac{1}{2}\sgL2\right)\frac{\partial u}{\partial y}, \quad y < \eta^*, \label{S1bis}\\
\frac{\partial u}{\partial t} - \dot{\eta}\frac{\partial u}{\partial y} = \frac{1}{2}\sgH2 \frac{\partial^2 u}{\partial y^2} + \left(\delta -
\frac{1}{2}\sgH2\right)\frac{\partial u}{\partial y}, \quad y > \eta^*,  \label{S2bis}
\end{eqnarray}
where $\displaystyle \dot{\eta}=\frac{d\eta}{dt}$. At the fixed interface $y=\eta^*$, the following interface conditions are satisfied:
\begin{eqnarray*}
\left\{
\begin{aligned}
&\big[u(t,\cdot)\big]_{y=\eta^*} = \bigg[\frac{\partial u}{\partial y}(t,\cdot)\bigg ]_{y=\eta^*} =0,\\
&u(t,\eta^*)= \gamma e^{\eta(t)}. 
\label{S3bis-S4}
\end{aligned}
\right.
\end{eqnarray*}

Next, we define the perturbations
\begin{equation}
v(t,y)=u(t,y)-K(y), \qquad\;\, f(t)=\eta(t)-\eta^*.
\label{form-v-f}
\end{equation}
and write the system verified by the pair $(v,f)$. From \eqref{S1bis} and \eqref{S2bis}, we get the partial differential equations:
\begin{eqnarray}
\frac{\partial v}{\partial t}  = \frac{1}{2}\sgL2 \frac{\partial^2 v}{\partial y^2} + \left(\delta -
\frac{1}{2}\sgL2\right)\frac{\partial v}{\partial y} + f'\left(\frac{\partial v}{\partial y} + \frac{dK}{dy}\right) , \quad y < \eta^*, \label{S1v}\\
\frac{\partial v}{\partial t}  = \frac{1}{2}\sgH2 \frac{\partial^2 v}{\partial y^2} + \left(\delta -
\frac{1}{2}\sgH2\right)\frac{\partial v}{\partial y} + f'\left(\frac{\partial v}{\partial y} + \frac{dK}{dy}\right) , \quad y > \eta^*, \label{S2v}.
\end{eqnarray}
The following interface conditions are satisfied at $y=\eta^*$:
\begin{eqnarray}
\left\{
\begin{aligned}
&\big[v(t,\cdot)\big]_{y=\eta^*} = \bigg [\frac{\partial v}{\partial y}(t,\cdot)\bigg ]_{y=\eta^*} =0,\\
&v(t,\eta^*)=\gamma e^{\eta^*}(e^{f(t)}-1).
\end{aligned}
\right.
\label{S3v-S4v}
\end{eqnarray}
It is readily seen that the term with $f'$ contributes additional difficulties that we will overcome in the next subsections.

\subsection{Ansatz}
According to the method of \cite{BHL00}, we introduce the following ansatz:
\begin{eqnarray}
v(t,y) =f(t)K'(y) + w(t,y),
\label{ansatz}
\end{eqnarray}
where $w$ is the new unknown. We are going to derive a system verified by $w$.

Plugging \eqref{ansatz} into \eqref{S1v}-\eqref{S2v}, the system for $(f,w)$ reads:
\begin{eqnarray}
\frac{\partial w}{\partial t}  = \frac{1}{2}\sgL2 \frac{\partial^2 w}{\partial y^2} + \bigg (\delta -
\frac{1}{2}\sgL2\bigg )\frac{\partial w}{\partial y} + f'\bigg (fK''+\frac{\partial w}{\partial y}\bigg ), \quad y < \eta^*, \label{S1w}\\
\frac{\partial w}{\partial t}  = \frac{1}{2}\sgH2 \frac{\partial^2 w}{\partial y^2} + \bigg (\delta -
\frac{1}{2}\sgH2\bigg )\frac{\partial w}{\partial y} + f'\bigg (fK''+\frac{\partial w}{\partial y}\bigg ), \quad y > \eta^*. \label{S2w}
\end{eqnarray}

The interface conditions at $y=\eta^*$ demand more attention, see \cite{ABLZ20,BLZ20}.
\begin{enumerate}[label=(\arabic*), wide, labelwidth=!, labelindent=0pt]
\item [\rm (i)] As $\displaystyle v(t,y) =f(t)K'(y)+ w(t,y)$, it follows that
\begin{eqnarray*}
\big[v(t,\cdot) \big]_{y=\eta^*} =f(t)[K']_{y=\eta^*} +  \big[w(t,\cdot) \big]_{y=\eta^*}.
\end{eqnarray*}
Hence, $[w(t,\cdot)]_{y=\eta^*}=0$.
 \item [\rm (ii)]
Since  $\displaystyle v(t,\eta^*) =f(t)K'(\eta^*)+ w(t,\eta^*) = \gamma e^{\eta^*}(e^{f(t)}-1)$ and $K'(\eta^*)=\gamma e^{\eta^*}\bigg (1-\dfrac{2}{\delta^2_L}\bigg )$, it comes:
 \begin{eqnarray}
 w(t,\eta^*)= \gamma e^{\eta^*}(e^{f(t)}-1)-f(t)\gamma e^{\eta^*}\bigg (1-\frac{2\delta}{\sgL2}\bigg ).
 \label{interf2}
 \end{eqnarray}
 \item [\rm (iii)] Differentiating  formula \eqref{ansatz} for $y \neq \eta^*$ and, then, taking the limit as $y$ tends to $\eta^*$ from the left and from the right, it follows that
\begin{eqnarray*}
\bigg [\frac{\partial w}{\partial y}(t)\bigg ]_{y=\eta^*} = - f(t)[K'']_{y=\eta^*}.
\end{eqnarray*}
Now, twice differentiating formula \eqref{funct-K} and, then, taking the left- and right-limit at $y=\eta^*$ gives
so that
\begin{eqnarray}
[K'']_{y=\eta^*}=(\gamma e^{\eta^*}-1)\bigg (1-\frac{2\delta}{\sigma^2_H}\bigg )^2-\gamma e^{\eta^*}\bigg (1-\frac{2\delta}{\sigma^2_L}\bigg )^2.
\label{form-1}
\end{eqnarray}
Since
\begin{eqnarray*}
(\gamma e^{\eta^*}-1)\bigg (1-\frac{2\delta}{\sigma^2_H}\bigg )=\gamma e^{\eta^*}\bigg (1-\frac{2\delta}{\sigma^2_L}\bigg ),
\end{eqnarray*}
from formula \eqref{form-1} we deduce that
\begin{eqnarray*}
[K'']_{y=\eta^*}
=\gamma e^{\eta^*}\bigg (1-\frac{2\delta}{\sigma^2_L}\bigg )\frac{2\delta}{\sigma_H^2\sigma_L^2}(\sigma_H^2-\sigma_L^2).
\end{eqnarray*}

Hence,
\begin{eqnarray}
\bigg [\frac{\partial w}{\partial y}(t,\cdot)\bigg ]_{y=\eta^*} = - f(t)[K'']_{y=\eta^*} = f(t)\gamma e^{\eta^*}\bigg (1-\frac{2\delta}{\sgL2}\bigg ) \frac{2\delta}{\sgH2\sgL2}(\sgL2-\sgH2).\label{interf3}
	\end{eqnarray}	
\end{enumerate}

\begin{remark} {\rm In contrast to $u$ and $v$, $w$ is not continuously differentiable at $y=\eta^*$ (see \cite{BHL00}).}
\end{remark}

The next point is to establish a relation between $f(t)$ and $w(t,\eta^*)$.
According to formula \eqref{interf2}, let us define the function $\Phi:\R\to\R$, by setting
\begin{eqnarray*}
\Phi(x)=\gamma e^{\eta^*}\Big(e^x-1-x +\frac{2\delta}{\sgL2}x\Big),\qquad\;\,x\in\R.
	\end{eqnarray*}
Clearly, $\Phi$ vanishes at the origin and $\Phi'(0)= \gamma e^{\eta^*}\dfrac{2\delta}{\sgL2}$.
It is easy to show that the function $\Phi$ is increasing in $[x_0,+\infty)$ and $\Phi(x_0)<0$, where $x_0= \log\Big (1-\frac{2\delta}{\sgL2}\Big )<0$ (note that $1-\frac{2\delta}{\sgL2}>0$ due to condition \eqref{admissible}).
 Therefore, the restriction of
the function $\Phi$ to the interval $(x_0,+\infty)$ is invertible. Let us denote its inverse by $\Psi$, which is an increasing function defined on $(\Phi(x_0), +\infty)$ and vanishes at
the origin.
	
We may rewrite equation \eqref{interf2} in the more compact form
$w(t,\eta^*)=\Phi(f(t))$. Therefore, is $w$ is small enough in some convenient norm, we can write $f(t)$ in terms of $w$, via the formula
\begin{eqnarray}
f(t)=	\Psi(w(t,\eta^*)).
\label{invert}
\end{eqnarray}

\begin{remark} \label{remark_linear}
{\rm Since $\Psi$ is smooth in in $(\Phi(x_0),+\infty)$, using Taylor expansion centered at zero, we can write
\begin{eqnarray}
\Psi(w(t,\eta^*)) =	\Psi'(0)w(t,\eta^*) + R(w(t,\eta^*))=e^{-\eta^*}\frac{\sgL2}{2\gamma\delta}w(t,\eta^*)+R(w(t,\eta^*)),
\label{invert_linear}
\end{eqnarray}
where $R$ is the remainder.}
\end{remark}

\subsection{The fully nonlinear problem} Thanks to equation \eqref{invert} and Remark \ref{remark_linear}, $f(t)$ can be eliminated in the relation \eqref{interf3}, obtaining the equivalent condition	
\begin{eqnarray*}
\bigg[\frac{\partial w}{\partial y}(t,\cdot)\bigg]_{y=\eta^*} =\bigg(1-\frac{2\delta}{\sgL2}\bigg ) \frac{1}{\sgH2}(\sgL2-\sgH2)w(t,\eta^*)
+R(w(t,\eta^*))\gamma e^{\eta^*}\bigg (1-\frac{2\delta}{\sgL2}\bigg ) \frac{2\delta}{\sgH2\sgL2}(\sgL2-\sgH2) \label{interf3_bis}
\end{eqnarray*}

The last step in the method consists in getting rid of $f'$ in equations \eqref{S1w}-\eqref{S2w}.

From \eqref{invert}, it comes:
\begin{eqnarray}
f'(t)=\Psi'(w(t,\eta^*))\frac{\partial w}{\partial t}(t,\eta^*).\label{dot_f}
\end{eqnarray}
To compute $\displaystyle \frac{\partial w}{\partial t}(t,\eta^*)$, we take the trace of \eqref{S1w} at $y=\eta^*_-$ (or equivalently the trace of \eqref{S2w} at $\eta^*_+$) and get
\begin{equation}
\frac{\partial w}{\partial t}(t,\eta^*)  = \frac{1}{2}\sgL2 \frac{\partial^2 w}{\partial y^2}(t,\eta^*_-) + \bigg (\delta -
\frac{1}{2}\sgL2\bigg )\frac{\partial w}{\partial y}(t,\eta^*_-) + f'(t)\bigg (\Psi(w(t,\eta^*))K''(\eta^*_-) + \frac{\partial w}{\partial y}(t,\eta^*_-)\bigg ).
\label{trace_eqn}
\end{equation}

Let us multiply \eqref{trace_eqn} by $\Psi'(w(t,\eta^*))$. Taking \eqref{dot_f} into account, we deduce that
\begin{align*}
f'(t) = &\frac{1}{2}\Psi'(w(t,\eta^*))\sgL2 \frac{\partial^2 w}{\partial y^2}(t,\eta^*_-) + \Psi'(w(t,\eta^*))\left(\delta -
\frac{1}{2}\sgL2\right)\frac{\partial w}{\partial y}(t,\eta^*_-) \\
&+ \Psi'(w(t,\eta^*))f'(t)\Big\{\Psi(w(t,\eta^*))K''(\eta^*_-) + \frac{\partial w}{\partial y}(t,\eta^*_-)\Big\},
\end{align*}
or, equivalently,
\begin{align*}
&f'(t)\bigg (1-\Psi'(w(t,\eta^*))\Big\{\Psi(w(t,\eta^*))K''(\eta^*_-) + \frac{\partial w}{\partial y}(t,\eta^*_-)\Big\}\bigg )\\
=& \frac{1}{2}\Psi'(w(t,\eta^*))\sgL2 \frac{\partial^2 w}{\partial y^2}(t,\eta^*_-) + \Psi'(w(t,\eta^*))\bigg (\delta -
\frac{1}{2}\sgL2\bigg )\frac{\partial w}{\partial y}(t,\eta^*_-).
\end{align*}
Assuming that $w$ is small enough (in some norm), we get to the following formula for the velocity of the perturbation of the free interface
\begin{eqnarray}
f'(t)=\frac{\frac{1}{2}\Psi'(w(t,\eta^*))\sgL2 \frac{\partial^2 w}{\partial y^2}(t,\eta^*_-) + \Psi'(w(t,\eta^*))\left(\delta -
\frac{1}{2}\sgL2\right)\frac{\partial w}{\partial y}(t,\eta^*_-)}
	{1-\Psi'(w(t,\eta^*))\big\{\Psi(w(t,\eta^*)) K''(\eta^*_-) + \frac{\partial w}{\partial y}(t,\eta^*_-)\big\}}.\label{velocity}
\end{eqnarray}
\begin{remark}
{\rm Formula \eqref{velocity} is an implicit \textit{second-order Stefan condition}, see \cite{BL18}.}
\end{remark}

Now, we define the linear differential operator $\LL$ and the boundary operator $B_1$ by setting
\begin{equation}\label{LL}
\LL w = \left\{
\begin{array}{ll}
 \frac{1}{2}\sgL2 \dfrac{\partial^2 w}{\partial y^2} + \left(\delta -
\frac{1}{2}\sgL2\right)\dfrac{\partial w}{\partial y}, \quad y < \eta^*,\\[3mm]
\frac{1}{2}\sgH2 \dfrac{\partial^2 w}{\partial y^2} + \left(\delta -
\frac{1}{2}\sgH2\right)\dfrac{\partial w}{\partial y}, \quad y > \eta^*,
\end{array}
\right.
\end{equation}
and
\begin{eqnarray}\label{BB}
B_1w = \bigg[\frac{\partial w}{\partial y}\bigg ]_{y=\eta^*}  - Aw(\eta^*), \qquad\;\, {\rm where}\;\, A= \left(1-\frac{2\delta}{\sgL2}\right) \frac{1}{\sgH2}(\sgL2-\sgH2)>0.
\end{eqnarray}

Using such operators, we can rewrite the fully nonlinear parabolic problem in the form:
\begin{equation}\label{FNLP}
\left\{
\begin{array}{ll}
\displaystyle \frac{d w}{d t}(t)= \LL w(t) + \FF(w(t)),\\[3mm]
B_1 w(t)= \GG(w(t)),
\end{array}
\right.
\end{equation}
where the nonlinear (quadratic) terms ${\mathcal F}$ and ${\mathcal G}$ are defined by
\begin{align}
\FF(w(t,y))&=\frac{\frac{1}{2}\Psi'(w(t,\eta^*))\sgL2 \frac{\partial^2 w}{\partial y^2}(t,\eta^*_-) + \Psi'(w(t,\eta^*))\left(\delta -
	\frac{1}{2}\sgL2\right)\frac{\partial w}{\partial y}(t,\eta^*_-)}
{1-\Psi'(w(t,\eta^*))\big\{\Psi(w(t,\eta^*))K''(\eta^*_-) + \frac{\partial w}{\partial y}(t,\eta^*_-)\big\}} \nonumber\\
&\qquad\quad\;\, \times \Big(\Psi(w(t,\eta^*))K''(y) + \frac{\partial w}{\partial y}(t,y)\Big),\label{Fscript}
\\
\GG(w(t))&=R(w(t,\eta^*))\gamma e^{\eta^*}\left(1-\frac{2\delta}{\sgL2}\right) \frac{2\delta}{\sgH2\sgL2}(\sgL2-\sgH2),
\notag
\end{align}
where
$K''(\eta^*_-)=\gamma e^{\eta^*}\left (1-\dfrac{2\delta}{\sigma_L^2}\right )^2$.

\begin{remark}
{\rm The problem for $w$ does not contain the free interface which has been eliminated. It is a \textit{fully nonlinear parabolic problem} (see \cite{Lunardi96}) because the nonlinear term \eqref{Fscript} contains traces of second-order derivatives with respect to the spatial variable $y$ and, therefore, it is of the same order as the linear operator $\LL$.}
\end{remark}

\section{Analysis of the linear problem}\label{analysis}
For notational convenience, we set
\begin{align*}
c_H=\frac{2\delta}{\sigma_H^2},\qquad\;\,c_L=\frac{2\delta}{\sigma_L^2}.
\end{align*}
For financial motivations (see \cite{LWH16}), the domain $(\AA)$ of admissible parameters is more specifically
\begin{align}\label{admissible_bis}
c_L<1<c_H\le 3.
\end{align}

In this section, we will prove that the realization of the operator ${\mathcal L}$ in a suitable weighted space of continuous functions generates an analytic semigroup. We also characterize its spectrum.
For this purpose, we introduce the weights $q_H:[\eta^*,+\infty)\to\R$ and $q_L:(-\infty,\eta^*]\to\R$, defined as follows:
\begin{align*}
&q_L(y)=\exp\bigg (-\frac{1}{2}(c_L-1)y\bigg ),\qquad\;\,y\in (-\infty,\eta^*],\\[1mm]
&q_H(y)=\exp\bigg (-\frac{1}{2}(c_H-1)y\bigg ),\qquad\;\,y\in [\eta^*,+\infty),
\end{align*}

\begin{definition}\label{H_norm}
We denote by $\mathscr{X}$ the space of all the pairs $\bm{w}=(w_1,w_2)$, where $w_1:(-\infty,\eta^*]\to\C$ and $w_2:[\eta^*,+\infty)\to\C$ are continuous functions such that
\begin{eqnarray*}
\|\bm{w}\|_{\mathscr{X}}:=\sup_{y\le\eta^*}\bigg |\frac{w_1(y)}{q_L(y)}\bigg |+\sup_{y\ge\eta^*}\bigg |\frac{w_2(y)}{q_H(y)}\bigg |<+\infty
\end{eqnarray*}

We also introduce the subspace $\mathscr{X}_2$ of $\mathscr{X}$ of pairs $\bm{w}=(w_1,w_2)$ such that $w_1\in C^2((-\infty,\eta^*];\C)$ and $w_2\in C^2([\eta^*,+\infty);\C)$ and
both $D_y\bm{w}:=(D_yw_1,D_yw_2)$ and $D_{yy}\bm{w}:=(D_{yy}w_1,D_{yy}w_2)$ belong to ${\mathscr X}$. It is normed by setting
\begin{eqnarray*}
\|\bm{w}\|_{\mathscr{X}_2}:=\|\bm{w}\|_{{\mathscr X}}+\|D_y\bm{w}\|_{\mathscr{X}}+\|D_{yy}\bm{w}\|_{\mathscr{X}}.
\end{eqnarray*}

\end{definition}

\begin{remark}
{\rm From the definition, it turns out that if $\bm{w}$ belongs to $\mathscr{X}$ and $c_L<1<c_H$, then $w_1$ and $w_2$ decrease to zero, with exponential rate, at $-\infty$ and $+\infty$, respectively.}
\end{remark}

\begin{lemma}\label{lemma-char-eq}
The space ${\mathscr X}_2$ can be equivalently characterized as the set of all pairs $\bm{w}=(w_1,w_2)$ such that the functions $\dfrac{w_1}{q_L}$ and $\dfrac{w_2}{q_H}$ belong to $C^2_b((-\infty,\eta^*];\C)$ and
$C^2_b([\eta^*,+\infty);\C)$, respectively. Moreover, the norm of ${\mathscr X}_2$ is equivalent to the classical norm of $C^2_b((-\infty,\eta^*];\C)\times C^2_b([\eta^*,+\infty);\C)$.
\end{lemma}

\begin{proof} A simple computation shows that
\begin{align*}
&D_y\bigg (\frac{w_1}{q_L}\bigg )=\frac{D_yw_1}{q_L}+\frac{1}{2}(c_L-1)\frac{w_1}{q_L},\\
&D_{yy}\bigg (\frac{w_1}{q_L}\bigg )=\frac{D_{yy}w_1}{q_L}+(c_L-1)\frac{D_yw_1}{q_L}+\frac{1}{4}(c_L-1)^2\frac{D_yw_1}{q_L}
\end{align*}
in $(-\infty,\eta^*]$. Hence, $\displaystyle\frac{w_1}{q_L}$ belongs to $C^2_b((-\infty,\eta^*];\C)$ and
\begin{align*}
\bigg\|\frac{w_1}{q_L}\bigg\|_{C^2_b((-\infty,\eta^*])}\le c_L\sup_{y\le\eta^*}\bigg |\frac{D_yw_1(y)}{q_L(y)}\bigg |+\frac{1}{4}(1-c_L)(3-c_L)\sup_{y\le\eta^*}\bigg |\frac{w_1(y)}{q_L(y)}\bigg |
+\sup_{y\le\eta^*}\bigg |\frac{D_{yy}w_1(y)}{q_L(y)}\bigg |.
\end{align*}

Similarly, $\frac{w_2}{q_H}$ belongs to $C^2_b([\eta^*,+\infty);\C)$ and
\begin{align*}
\bigg\|\frac{w_2}{q_H}\bigg\|_{C^2_b([\eta^*,+\infty))}\le c_H\sup_{y\ge\eta^*}\bigg |\frac{D_yw_2(y)}{q_H(y)}\bigg |+\frac{1}{4}(c_H^2-1)\sup_{y\ge\eta^*}\bigg |\frac{w_2(y)}{q_H(y)}\bigg |
+\sup_{y\ge\eta^*}\bigg |\frac{D_{yy}w_2(y)}{q_L(y)}\bigg |.
\end{align*}

Combining these last two estimates, we can infer that there exists a positive constant $C$, independent of $\bm{w}$ such that
\begin{eqnarray*}
\bigg\|\frac{w_1}{q_L}\bigg\|_{C^2_b((-\infty,\eta^*])}+\bigg\|\frac{w_2}{q_H}\bigg\|_{C^2_b([\eta^*,+\infty)}\le C\|\bm{w}\|_{\mathscr{X}_2}.
\end{eqnarray*}

Vice versa, suppose that $z_1\in C^2_b((-\infty,\eta^*];\C)$ and $z_2 \in C^2_b([\eta^*,+\infty);\C)$ and set $\bm{w}=(w_1,w_2)$, where $w_1=z_1q_L$ and $w_2=z_2q_H$. Clearly,
$\bm{w}$ belongs to $\mathscr{X}$ and $\|\bm{w}\|_{\mathscr{X}}=\|z_1\|_{C_b((-\infty,\eta^*];\C)}+\|z_2\|_{C_b([\eta^*,+\infty);\C)}$.
Moreover, since
\begin{align*}
D_yw_1=\bigg [D_yz_1-\frac{1}{2}(c_L-1)z_1\bigg ]q_L,\qquad\;\, D_{yy}w_1=\bigg [D_{yy}z_1+(1-c_L)D_yz_1+\frac{1}{4}(c_L-1)^2z_1\bigg ]q_L
\end{align*}
in $(-\infty,\eta^*]$ and
\begin{align*}
D_yw_2=\bigg [D_yz_2-\frac{1}{2}(c_H-1)z_2\bigg ]q_H,\qquad\;\, D_{yy}w_2=\bigg [D_{yy}z_2-(c_H-1)D_yz_2+\frac{1}{4}(c_H-1)^2z_2\bigg ]q_H
\end{align*}
in $[\eta^*,+\infty)$, it is an easy task to check that $\bm{w}\in\mathscr{X}_2$ and
\begin{eqnarray*}
\|\bm{w}\|_{\mathscr{X}_2}\le C_2\big (\|z_1\|_{C^2_b((-\infty,\eta^*];\C)}+\|z_2\|_{C^2_b([\eta^*,+\infty);\C)}\big )
\end{eqnarray*}
for some positive constant $C_2$, independent of $\bm{w}$. This completes the proof.
\end{proof}

\subsection{The generation result}
This the main result of this subsection.

\begin{theorem}
\label{thm:gen}
The realization $L$ of the operator $\mathcal{L}$ in $\mathscr{X}$, with domain
\begin{eqnarray*}
D(L)=\{\bm{w}\in \mathscr{X}: {\mathcal L}\bm{w}\in {\mathscr X},\ {\mathcal B}(\bm{w})=\bm{0}\}=\{\bm{w}\in\mathscr{X}_2: {\mathcal B}(\bm{w})=\bm{0}\},
\end{eqnarray*}
generates an analytic semigroup. Here, ${\mathcal B}\bm{w}=(B_0\bm{w},B_1\bm{w})$, where $B_0\bm{w}=w_2(\eta^*)-w_1(\eta^*)$ and the operator $B_1$ is defined by \eqref{BB}.
Moreover, its spectrum $\sigma(L)$ consists of the union of the halfline $\left (-\infty, -\min\left\{\frac{c_L^2}{8}(c_L^2-1)^2,\frac{c_H^2}{8}(c_H^2-1)^2\right\}\right ]$ and the solutions of the dispersion relation ${\mathcal D}_{\lambda}=0$, where
\begin{eqnarray*}
{\mathcal D}_{\lambda}=A-\mu^H_{-}+\mu^L_{+},
\end{eqnarray*}
and
\begin{align*}
&\mu^J_{\pm}=-\frac{1}{2}(c_J-1)\pm\sqrt{\frac{1}{4}(c_J-1)^2+\frac{2\lambda}{\sigma_J^2}},\qquad\;\,J\in\{H,L\}.
\end{align*}
\end{theorem}

\begin{proof}
Let us fix $\bm{f}\in\mathscr{X}$, $\lambda\in\mathbb C$ and consider the resolvent equation $\lambda\bm{w}-\mathscr{L}\bm{w}=\bm{f}$. We need to distinguish some cases.

{\it Case 1}. If $\lambda\notin \left (-\infty,-\frac{c_L^2}{8}(c_L^2-1)^2\right ]\cup \left (-\infty,-\frac{c_H^2}{8}(c_H^2-1)^2\right ]$, then
an easy computation reveals that the more general solution $\bm{w}_{\lambda}$ of such an equation, which belongs to the space $\mathscr{X}$, is given by
\begin{align}
w_{1,\lambda}(y)=&\bigg (c_{1,\lambda}+\frac{1}{\mu^L_+-\mu^L_-}\int_y^{\eta^*}e^{-\mu^L_+s}f_1(s)ds\bigg )e^{\mu^L_+y}
+\frac{1}{\mu^L_+-\mu^L_-}\bigg (\int_{-\infty}^ye^{-\mu^L_-s}f_1(s)ds\bigg )e^{\mu^L_-y}
\label{w1}
\end{align}
for every $y\in (-\infty,\eta^*]$ and
\begin{align}
w_{2,\lambda}(y)=&\frac{1}{\mu^H_+-\mu^H_-}\bigg (\int_y^{+\infty}e^{-\mu^H_+s}f_2(s)ds\bigg )e^{\mu^H_+y}
+\bigg (d_{1,\lambda}+\frac{1}{\mu^H_+-\mu^H_-}\int_{\eta^*}^ye^{-\mu^H_-s}f_2(s)ds\bigg )e^{\mu^H_-y}
\label{w2}
\end{align}
for every $y\in [\eta^*,+\infty)$.

Let us now impose the boundary conditions at $\eta^*$. The first condition is $B_0\bm{w}_{1,\lambda}=0$, i.e., $w_{1,\lambda}(\eta^*)=w_{2,\lambda}(\eta^*)$, which means that
\begin{align}
c_{1,\lambda}e^{\mu^L_+\eta^*}-d_{1,\lambda}e^{\mu^H_-\eta^*}=&-
\frac{e^{\mu^L_-\eta^*}}{\mu^L_+-\mu^L_-}\int_{-\infty}^{\eta^*}e^{-\mu^L_-s}f_1(s)ds
+\frac{e^{\mu^H_+\eta^*}}{\mu^H_+-\mu^H_-}\int_{\eta^*}^{+\infty}e^{-\mu^H_+s}f_2(s)ds,
\label{dirty}
\end{align}
whereas the second condition $B_1\bm{w}_{\lambda}=0$, i.e.,
$\displaystyle\frac{\partial w_2}{\partial y}(\eta^*)-\displaystyle\frac{\partial w_1}{\partial y}(\eta^*)-Aw_2(\eta^*)=0$, reads as follows:
\begin{align}
&\mu^L_+e^{\mu^L_+\eta^*}c_{1,\lambda}+(A-\mu^H_-)e^{\mu^H_-\eta^*}d_{1,\lambda}\notag\\
=&\frac{\mu^H_+-A}{\mu^H_+-\mu^H_-}\bigg (\int_{\eta^*}^{+\infty}e^{-\mu^H_+s}f_2(s)ds\bigg )e^{\mu^H_+\eta^*}-\frac{\mu^L_-}{\mu^L_+-\mu^L_-}\bigg (\int_{-\infty}^{\eta^*}e^{-\mu^L_-s}f_1(s)ds\bigg )e^{\mu^L_-\eta^*}.
\label{dancing}
\end{align}

The system \eqref{dirty}-\eqref{dancing} is uniquely solvable if and only if and only if ${\mathcal D}_{\lambda}\neq 0$, i.e., if and only if
\begin{align*}
(1-c_L)\bigg (\frac{\sgL2}{\sgH2}-1\bigg )+\frac{1}{2}(c_H+c_L)-1+\sqrt{\frac{1}{4}(c_H-1)^2+\frac{2\lambda}{\sigma_H^2}}
-\sqrt{\frac{1}{4}(c_L-1)^2+\frac{2\lambda}{\sigma_L^2}}\neq 0.
\end{align*}
In this case,
\begin{align*}
c_{1,\lambda}=&
-\frac{A-\mu^H_{-}+\mu^L_-}{(\mu^L_+-\mu^L_-){\mathcal D}_{\lambda}}\bigg (\int_{-\infty}^{\eta^*}e^{-\mu^L_-s}f_1(s)ds\bigg )e^{(\mu^L_--\mu^L_+)\eta^*}+\frac{1}{{\mathcal D}_{\lambda}}\bigg (\int_{\eta^*}^{+\infty}e^{-\mu^H_+s}f_2(s)ds\bigg )e^{(\mu^H_+-\mu^L_+)\eta^*}
\end{align*}
and
\begin{align*}
d_{1,\lambda}=
\frac{1}{{\mathcal D}_{\lambda}}&\bigg (\int_{-\infty}^{\eta^*}e^{-\mu^L_-s}f_1(s)ds\bigg )e^{(\mu^L_--\mu^H_-)\eta^*}
-\frac{A-\mu^H_++\mu^L_+}{(\mu^H_+-\mu^H_-){\mathcal D}_{\lambda}}\bigg (\int_{\eta^*}^{+\infty}e^{-\mu^H_+s}f_2(s)ds\bigg )e^{(\mu^H_+-\mu^H_-)\eta^*},
\end{align*}

On the other hand, if ${\mathcal D}_{\lambda}=0$, then the equation $\lambda\bm{w}-{\mathcal L}\bm{w}=\bm{0}$, subject to the boundary conditions ${\mathcal B}\bm{w}={\bm 0}$, admits more than a unique solution, i.e., every
solution to the dispersion relation ${\mathcal D}_{\lambda}=0$ lies in the point spectrum of the operator $({\mathcal L},{\mathcal B})$.

Suppose now that $\lambda\in \left (-\infty,-\frac{c_L^2}{8}(c_L^2-1)^2\right ]$. We claim that $\lambda$ belongs to the spectrum of the operator $L$. For this purpose, we observe that, if we take $\bm{f}=\bm{0}$, then
the more general solution to the equation $\lambda\bm{w}-{\mathcal L}\bm{w}=\bm{f}$, which belongs to ${\mathscr X}$, has the first component which is given by
\begin{eqnarray*}
w_1(y)=c_1e^{\mu^L_+y}+c_2e^{\mu^L_-y},\qquad\;\,y\in (-\infty,\eta^*],
\end{eqnarray*}
whereas the second component is given by
\begin{eqnarray*}
w_2(y)=d_1e^{\mu^L_-y},\qquad\;\,y\in (-\infty,\eta^*],
\end{eqnarray*}
if $\lambda\notin \left (-\infty, -\frac{c_H^2}{8}(c_H^2-1)^2\right ]$, and
\begin{eqnarray*}
w_2(y)=d_1e^{\mu^L_-y}+d_2e^{\mu^L_+y},\qquad\;\,y\in (-\infty,\eta^*],
\end{eqnarray*}
if $\lambda\in \left (-\infty, -\frac{c_H^2}{8}(c_H^2-1)^2\right ]$. Therefore, we have at least three parameters to identify and two boundary conditions. Clearly, this is not possible. Hence, the interval
$\left (-\infty, -\frac{c_L^2}{8}(c_L^2-1)^2\right ]$ is contained in the spectrum of the operator $L$.

A completely similar argument shows that also the interval $\left (-\infty, -\frac{c_H^2}{8}(c_H^2-1)^2\right ]$ is contained in $\sigma(L)$.

Now, we observe that the function  $\bm{w}_{\lambda}$, defined by \eqref{w1} and \eqref{w2}, has first- and second-order derivatives which belong to ${\mathscr X}$.
Indeed,
\begin{align*}
\frac{\partial w_{1,\lambda}}{\partial y}(y)=&\mu^L_+\bigg (c_{1,\lambda}+\frac{1}{\mu^L_+-\mu^L_-}\int_y^{\eta^*}e^{-\mu^L_+s}f_1(s)ds\bigg )e^{\mu^L_+y}
+\frac{\mu^L_-}{\mu^L_+-\mu^L_-}\bigg (\int_{-\infty}^ye^{-\mu^L_-s}f_1(s)ds\bigg )e^{\mu^L_-y}
\end{align*}
for every $y\in (-\infty,\eta^*]$ and
\begin{align*}
\frac{\partial w_{2,\lambda}}{\partial y}(y)=&\frac{\mu^H_+}{\mu^H_+-\mu^H_-}\bigg (\int_y^{+\infty}e^{-\mu^H_+s}f_2(s)ds\bigg )e^{\mu^H_+y}
+\mu^H_-\bigg (d_{1,\lambda}\!+\!\frac{1}{\mu^H_+-\mu^H_-}\int_{\eta^*}^ye^{-\mu^H_-s}f_2(s)ds\bigg )e^{\mu^H_-y}
\end{align*}
for every $y\in [\eta^*,+\infty)$.
These two functions are clearly continuous in $(-\infty,\eta^*]$ and in $[\eta^*,+\infty)$, respectively. Moreover, it is an easy task to show that
the functions $\displaystyle\frac{w_{1,\lambda}}{q_L}$ and $\displaystyle\frac{w_{2,\lambda}}{q_H}$ are bounded in $(-\infty,\eta^*]$ and in $[\eta^*,+\infty)$, respectively.
Since
\begin{align*}
&\frac{\partial^2w_{1,\lambda}}{\partial y^2}=\frac{2}{\sigma_L^2}\bigg (\lambda w_{1,\lambda}-\bigg (\delta-\frac{1}{2}\sigma_L^2\bigg )D_yw_{1,\lambda}(y)-f_1(y)\bigg ),\qquad\;\,y\in (-\infty,\eta^*],\\
&\frac{\partial^2w_{2,\lambda}}{\partial y^2}=\frac{2}{\sigma_H^2}\bigg (\lambda w_{2,\lambda}-\bigg (\delta-\frac{1}{2}\sigma_H^2\bigg )D_yw_{2,\lambda}(y)-f_1(y)\bigg ),\qquad\;\,y\in [\eta^*,+\infty)
\end{align*}
and $\bm{w}_{\lambda}$, $D_y\bm{w}_{\lambda}$ and $\bm{f}$ belong to ${\mathscr X}$, by difference we conclude that $D_{yy}\bm{w}\in {\mathscr X}$.

The characterization of the domain of $L$ and of its spectrum is complete.

To conclude the proof, we need to show that $L$ is a sectorial operator. For this purpose, we set $R(\lambda,L)\bm{f}:=\bm{w}_{\lambda}$ for every $\lambda\in\rho(L)$ and show that
there exist two positive constants $C$ and $M$ such that
\begin{equation}
\|R(\lambda,L)\bm{f}\|_{L(\mathscr{X})}\le M|\lambda|^{-1}\|\bm{f}\|_{\mathscr{X}},
\label{idro-11}
\end{equation}
for every $\lambda\in\C$, with ${\rm Re}\lambda\ge M$. In view of \cite[Proposition 3.2.8]{abdel-luca} this will be enough to conclude that the operator $L$ is sectorial in ${\mathscr H}$.

We begin by estimating the function $w_{1,1,\lambda}:(-\infty,\eta^*]\to\C$, defined by
\begin{align*}
w_{1,1,\lambda}(y)=\frac{1}{\mu^L_+-\mu^L_-}\bigg (\int_y^{\eta^*}e^{-\mu^L_+s}f_1(s)ds\bigg )e^{\mu^L_+y}
+\frac{1}{\mu^L_+-\mu^L_-}\bigg (\int_{-\infty}^ye^{-\mu^L_-s}f_1(s)ds\bigg )e^{\mu^L_-y}
\end{align*}
for every $y\in (-\infty,\eta^*]$. Observe that
\begin{align*}
\bigg |\frac{w_{1,1,\lambda}(y)}{q_L(y)}\bigg |\le &\frac{1}{|\mu^L_+-\mu^L_-|}\sup_{s\le\eta}\bigg |\frac{f_1(s)}{q_L(s)}\bigg |\bigg (
\int_y^{\eta^*}e^{\frac{1}{2}(\mu^L_+-\mu^L_-)(y-s)}ds
+\int_{-\infty}^ye^{-\frac{1}{2}(\mu^L_+-\mu^L_-)(y-s)}ds\bigg )\\
\le & \frac{2}{|\mu^L_+-\mu^L_-|}\sup_{s\le\eta}\bigg |\frac{f_1(s)}{q_L(s)}\bigg |\bigg (\int_0^{\eta^*-y}e^{-\frac{1}{2}{\rm Re}(\mu^L_+-\mu^L_-)s}ds+\int_0^{+\infty}e^{-\frac{1}{2}{\rm Re}(\mu^L_+-\mu^L_-)s}ds\bigg )\\
= &\frac{2}{|\mu^L_+-\mu^L_-|}\frac{1}{{\rm Re}(\mu^L_+-\mu^L_-)}\sup_{s\le\eta}\bigg |\frac{f_1(s)}{q_L(s)}\bigg |
\big (2-e^{-\frac{1}{2}{\rm Re}(\mu^L_+-\mu^L_-)(\eta^*-y)}\big )\\
\le &\frac{4}{|\mu^L_+-\mu^L_-|}\frac{1}{{\rm Re}(\mu^L_+-\mu^L_-)}\sup_{s\le\eta}\bigg |\frac{f_1(s)}{q_L(s)}\bigg |
\end{align*}
for every $y\in (-\infty,\eta^*]$. Since
\begin{eqnarray*}
\mu^L_+-\mu^L_-=\sqrt{(c_L-1)^2+\frac{8\lambda}{\sigma_L^2}}
\end{eqnarray*}
it follows easily that
\begin{align*}
|\mu^L_+-\mu^L_-|=&\sqrt{|\lambda|}\sqrt{\bigg |\frac{(c_L-1)^2}{\lambda}+\frac{8}{\sigma_L^2}\bigg |}
\ge \sqrt{|\lambda|}\sqrt{\frac{8}{\sigma_L^2}-\frac{(c_L-1)^2}{|\lambda|}}
\ge \frac{2}{\sigma_L}\sqrt{|\lambda|},
\end{align*}
provided that $|\lambda|\ge \frac{\sigma_L^2(c_L-1)^2}{4}$.

Moreover, recalling that by $\sqrt{\cdot}$ we mean the main determination of the square root in $\C$ and ${\rm Re}\sqrt{z}=\frac{\sqrt{2}}{2}\sqrt{|z|+{\rm Re}(z)}$ for every $z\in\C$, we can also estimate
\begin{align}
{\rm Re}(\mu^L_+-\mu^L_-)\ge\frac{\sqrt{2}}{2}\sqrt{\bigg |(c_L-1)^2+\frac{8\lambda}{\sigma_L^2}\bigg |+(c_L-1)^2+\frac{8{\rm Re}(\lambda)}{\sigma_L^2}}
\ge\frac{\sqrt{2}}{2}\sqrt{\bigg |(c_L-1)^2+\frac{8\lambda}{\sigma_L^2}\bigg |},
\label{idro-4}
\end{align}
if the real part of $\lambda$ is nonnegative. Summing up, if ${\rm Re}(\lambda)\ge \frac{\sigma_L^2(c_L-1)^2}{4}$, then
\begin{align*}
\sup_{y\le\eta^*}\bigg |\frac{w_{1,1,\lambda}(y)}{q_L(y)}\bigg |\le \frac{\sqrt{2}\sigma_L}{|\lambda|}\sup_{s\le\eta}\bigg |\frac{f_1(s)}{q_L(s)}\bigg |.
\end{align*}

In a completely similar way, we can estimate the function $w_{2,1,\lambda}:[\eta^*,+\infty)\to\C$, defined by
\begin{align*}
w_{2,1,\lambda}(y)=\frac{1}{\mu^H_+-\mu^H_-}\bigg (\int_y^{+\infty}e^{-\mu^H_+s}f_2(s)ds\bigg )e^{\mu^H_+y}
+\frac{1}{\mu^H_+-\mu^H_-}\bigg (\int_{\eta^*}^ye^{-\mu^H_-s}f_2(s)ds\bigg )e^{\mu^H_-y}
\end{align*}
for every $y\in [\eta^*,+\infty)$. It turns out that
\begin{align*}
\sup_{y\le\eta^*}\bigg |\frac{w_{2,1,\lambda}(y)}{q_H(y)}\bigg |\le \frac{\sqrt{2}\sigma_H}{|\lambda|}\sup_{s\ge\eta}\bigg |\frac{f_2(s)}{q_L(s)}\bigg |.
\end{align*}

We now turn our attention to the function $y\mapsto w_{1,2,\lambda}(y)=c_{1,\lambda}e^{\mu^L_+y}$ and note that
\begin{align}
\bigg |\frac{w_{1,2,\lambda}(y)}{q_L(y)}\bigg |\le |c_{1,\lambda}|e^{\frac{1}{2}{\rm Re}(\mu^L_+-\mu^L_-)\eta_*},\qquad\;\,y\in (-\infty,\eta^*].
\label{idro-8}
\end{align}
Hence, we need to estimate the constant $c_{1,\lambda}$. For this purpose, we observe that
\begin{align}
|c_{1,\lambda}|\le &\frac{|A-\mu^H_{-}+\mu^L_-|}{|\mu^L_+-\mu^L_-||{\mathcal D}_{\lambda}|}\bigg |\int_{-\infty}^{\eta^*}e^{-\mu^L_-s}f_1(s)ds\bigg |e^{{\rm Re}(\mu^L_--\mu^L_+)\eta^*}\notag\\
&+\frac{1}{|{\mathcal D}_{\lambda}|}\bigg |\int_{\eta^*}^{+\infty}e^{-\mu^H_+s}f_2(s)ds\bigg |e^{{\rm Re}(\mu^H_+-\mu^L_+)\eta^*}\notag\\
\le & \frac{\sigma_L|A-\mu^H_{-}+\mu^L_-|}{2\sqrt{|\lambda|}|{\mathcal D}_{\lambda}|}\sup_{s\le\eta}\bigg |\frac{f_1(s)}{q_L(s)}\bigg |\bigg (\int_{-\infty}^{\eta^*}e^{\frac{1}{2}{\rm Re}(\mu^L_+-\mu^L_-)s}ds\bigg )e^{{\rm Re}(\mu^L_--\mu^L_+)\eta^*}\notag\\
&+\frac{1}{|{\mathcal D}_{\lambda}|}\sup_{s\ge\eta}\bigg |\frac{f_2(s)}{q_H(s)}\bigg |\bigg (\int_{\eta^*}^{+\infty}e^{-\frac{1}{2}{\rm Re}(\mu^H_+-\mu^H_-)s}ds\bigg )
e^{{\rm Re}(\mu^H_+-\mu^L_+)\eta^*}\notag\\
\le & \frac{\sigma_L|A-\mu^H_{-}+\mu^L_-|}{\sqrt{|\lambda|}|{\mathcal D}_{\lambda}|}\frac{1}{{\rm Re}(\mu^L_+-\mu^L_-)}\sup_{s\le\eta}\bigg |\frac{f_1(s)}{q_L(s)}\bigg |
e^{-\frac{1}{2}{\rm Re}(\mu^L_+-\mu^L_+)\eta^*}\notag\\
&+\frac{2}{|{\mathcal D}_{\lambda}|}\frac{1}{{\rm Re}(\mu^H_+-\mu^H_-)}\sup_{s\ge\eta}\bigg |\frac{f_2(s)}{q_H(s)}\bigg |
e^{\frac{1}{2}{\rm Re}(\mu^H_++\mu^H_--2\mu^L_+)\eta^*}.
\label{idro-1}
\end{align}
Note that
\begin{align*}
{\mathcal D}_{\lambda}=&A-\mu^H_-+\mu^L_+\\
=&\bigg (1-\frac{2\delta}{\sigma_L^2}\bigg )\frac{1}{\sigma_H^2}
(\sigma^2_L-\sigma^2_H)+\frac{1}{2}(c_H-c_L)
+\sqrt{\frac{1}{4}(c_H-1)^2+\frac{2\lambda}{\sigma_H^2}}
+\sqrt{\frac{1}{4}(c_L-1)^2+\frac{2\lambda}{\sigma_L^2}}\\
=&(1-c_L)\bigg (\frac{c_H}{c_L}-1\bigg )+\frac{1}{2}(c_H-c_L)
+\sqrt{\frac{1}{4}(c_H-1)^2+\frac{2\lambda}{\sigma_H^2}}
+\sqrt{\frac{1}{4}(c_L-1)^2+\frac{2\lambda}{\sigma_L^2}}\\
=&\frac{c_H}{c_L}+\frac{1}{2}(c_L-c_H)-1+
\sqrt{\frac{1}{4}(c_H-1)^2+\frac{2\lambda}{\sigma_H^2}}
+\sqrt{\frac{1}{4}(c_L-1)^2+\frac{2\lambda}{\sigma_L^2}},
\end{align*}
so that
\begin{align}
|{\mathcal D}_{\lambda}|\ge &
\sqrt{|\lambda|}\bigg |\sqrt{\frac{1}{4\lambda}(c_H-1)^2+\frac{2}{\sigma_H^2}}
+\sqrt{\frac{1}{4\lambda}(c_L-1)^2+\frac{2}{\sigma_L^2}}\bigg |-\frac{c_H}{c_L}+\frac{1}{2}(c_L-c_H)-1\notag\\
\ge & C_1\sqrt{|\lambda|}
\label{idro-2}
\end{align}
for a positive constant $C_1$ and every $\lambda\in\C$ with sufficiently large real part. In a completely similar way, we can check that there exists a positive constant $C_2$ such that
\begin{align}
|A-\mu^H_-+\mu^L_-|\ge C_2\sqrt{|\lambda|}
\label{idro-3}
\end{align}
for every $\lambda\in\C$ with sufficiently large real part.

From \eqref{idro-1}-\eqref{idro-3} we deduce that
\begin{align*}
|c_{1,\lambda}|\le \frac{C_3}{|\lambda|}\bigg (\sup_{s\le\eta}\bigg |\frac{f_1(s)}{q_L(s)}\bigg |e^{-\frac{1}{2}{\rm Re}(\mu^L_+-\mu^L_-)\eta^*}
+\sup_{s\le\eta}\bigg |\frac{f_1(s)}{q_L(s)}\bigg |e^{\frac{1}{2}{\rm Re}(\mu^H_++\mu^H_--2\mu^L_+)\eta^*}\bigg )
\end{align*}
for every $\lambda\in\C$, with sufficiently large real part, and some positive constant $C_3$, independent of $\lambda$, and then that
\begin{align}
\bigg |\frac{w_{1,2,\lambda}(y)}{q_L(y)}\bigg |
\le &\frac{C_3}{|\lambda|}\bigg (\sup_{s\le\eta}\bigg |\frac{f_1(s)}{q_L(s)}\bigg |
+\sup_{s\le\eta}\bigg |\frac{f_1(s)}{q_L(s)}\bigg |e^{\frac{1}{2}{\rm Re}(\mu^H_++\mu^H_--\mu^L_+-\mu^L_=)\eta^*}\bigg )\notag\\
=&\frac{C_3}{|\lambda|}\bigg (\sup_{s\le\eta}\bigg |\frac{f_1(s)}{q_L(s)}\bigg |
+\sup_{s\le\eta}\bigg |\frac{f_1(s)}{q_L(s)}\bigg |e^{(c_H-c_L)\eta^*}\bigg )
\label{idro-5}
\end{align}
for every $y\in (-\infty,\eta^*]$ and the same values of $\lambda$ as above.
From \eqref{idro-4} and \eqref{idro-5} it follows that
\begin{align}
\sup_{y\le \eta^*}\bigg |\frac{w_{1,\lambda}(y)}{q_L(y)}\bigg |
\le \frac{C_4}{|\lambda|}\|\bm{f}\|_{{\mathscr X}}
\label{idro-6}
\end{align}
for every $\lambda\in\C$ with ${\rm Re}\lambda\ge M_1$. Here, $C_4$ and $M_1$ are positive constants, independent of $\lambda$.

To complete the estimate of the function $w_{2,\lambda}$, we need to consider the term $d_{1,\lambda}e^{\mu^H_-y}$. For this purpose, we observe that
\begin{align*}
|d_{1,\lambda}|\le &\frac{1}{|{\mathcal D}_{\lambda}|}\sup_{s\le\eta^*}\bigg |\frac{f_1(s)}{q_L(s)}\bigg|\bigg (\int_{-\infty}^{\eta^*}e^{\frac{1}{2}{\rm Re}(\mu^L_+-\mu^L_-)s}ds\bigg )
e^{{\rm Re}(\mu^L_--\mu^L_+)\eta^*}\\
&+\frac{|A-\mu^H_-+\mu^L_+|}{(\mu^H_+-\mu^H_-)|{\mathcal D}_{\lambda}|}\sup_{s\ge\eta^*}\bigg |\frac{f_2(s)}{q_H(s)}\bigg |
\bigg (\int_{\eta^*}^{+\infty}e^{-\frac{1}{2}{\rm Re}(\mu^H_+-
\mu^H_-)s}ds\bigg )e^{{\rm Re}(\mu^H_+-\mu^H_-)\eta^*}\\
\le &\frac{2}{|{\mathcal D}_{\lambda}|}\frac{1}{{\rm Re}(\mu^L_+-\mu^L_-)}\sup_{s\le\eta^*}\bigg |\frac{f_1(s)}{q_L(s)}\bigg |e^{-\frac{1}{2}{\rm Re}(\mu^L_+-\mu^L_-)\eta^*}\\
&+\frac{|A-\mu^H_-+\mu^L_+|}{(\mu^H_+-\mu^H_-)|{\mathcal D}_{\lambda}|}\frac{2}{{\rm Re}(\mu^H_+-\mu^H_-)}\sup_{s\ge\eta^*}\bigg |\frac{f_2(s)}{q_H(s)}\bigg |e^{\frac{1}{2}{\rm Re}(mu^H_+-\mu^H_-)\eta^*}.
\end{align*}

It follows that
\begin{align}
|d_{1,\lambda}|e^{-\frac{1}{2}{\rm Re}(\mu^H_+-\mu^H_-)y}
\le &\frac{2}{|{\mathcal D}_{\lambda}|}\frac{1}{{\rm Re}(\mu^L_+-\mu^L_-)}\sup_{s\le\eta^*}\bigg |\frac{f_1(s)}{q_L(s)}\bigg |e^{\frac{1}{2}{\rm Re}(\mu^L_++\mu^L_--\mu^H_+-\mu^H_-)\eta^*}\notag\\
&+\frac{|A-\mu^H_-+\mu^L_+|}{(\mu^H_+-\mu^H_-)|{\mathcal D}_{\lambda}|}\frac{2}{{\rm Re}(\mu^H_+-\mu^H_-)}\sup_{s\ge\eta^*}\bigg |\frac{f_2(s)}{q_H(s)}\bigg |\notag\\
\le & \frac{C_5}{|\lambda|}\|\bm{f}\|_{{\mathscr X}}
\label{idro-9}
\end{align}
for every $y\in [\eta^*,+\infty)$ and $\lambda\in\C$, with sufficiently large real part.
From \eqref{idro-8} and \eqref{idro-9} we deduce that
there exists a positive constant $C_6$ such that
\begin{equation}
\sup_{y\ge\eta^*}\bigg |\frac{f_2(y)}{q_H(y)}\bigg |\le C_6\|\bm{f}\|_{{\mathscr X}}
\label{idro-10}
\end{equation}
for every $\lambda\in\C$ with ${\rm Re}\lambda\ge M_2$ and some positive constants $C_6$, independent of $\lambda$, and $M_2$.

From \eqref{idro-6} and \eqref{idro-10}, estimate \eqref{idro-11} follows at once, with $M=\max\{M_1,M_2\}$. The proof is now complete.
\end{proof}

For further use, the following corollary will be particularly useful.

\begin{corollary}
For every $\lambda\in\rho(L)$, every $\bm{f}\in\mathscr{X}$, every $g\in\R$, the Cauchy problem
\begin{align}
\left\{
\begin{array}{l}
\lambda\bm{w}-\mathcal{L}\bm{w}=\bm{f},\\
{\mathcal B}\bm{w}=(0,a),
\end{array}
\right.
\label{boundary-pb-2}
\end{align}
admits a unique solution $\bm{w}\in {\mathscr X}_2$. Moreover, there exists a positive constant $C$, independent of the data and of $\lambda$, such that
\begin{equation}
|\lambda|\|\bm{w}\|_{\mathscr X}\le C(\|\bm{f}\|+\sqrt{|\lambda|}|a|)
\label{estimate-boundary}
\end{equation}
for $\lambda\in\C$ with sufficiently large real part.
\end{corollary}

\begin{proof}
In view of the results in Theorem \ref{thm:gen}, it suffices to deal with the Cauchy problem
\begin{align*}
\left\{
\begin{array}{l}
\lambda\bm{w}-\mathcal{L}\bm{w}=\bm{0},\\
{\mathcal B}\bm{w}=(0,a).
\end{array}
\right.
\end{align*}
Adapting the arguments used in the quoted theorem, it can be easily checked that the solution to the previous equation is given
by $\bm{w}=(w_1,w_2)$, where
\begin{align*}
w_1(y)&=\frac{a}{{\mathcal D}_{\lambda}}e^{\mu_L^+(y-\eta^*)},\qquad\;\,y\in (-\infty,\eta^*],\\
w_2(y)&=\frac{a}{{\mathcal D}_{\lambda}}e^{\mu_H^-(y-\eta^*)},\qquad\;\,y\in [\eta^*,+\infty).
\end{align*}

Clearly, $\bm{w}\in\mathscr{X}_2$ and taking \eqref{idro-2} into account, we can show that $\|\bm{w}\|_{\mathscr{X}}\le |\lambda|^{-\frac{1}{2}}|a|$.

Estimate \eqref{estimate-boundary} now follows from the previous estimate and \eqref{idro-11}, observing that the solution to the Cauchy problem \eqref{boundary-pb-2} is given by $R(\lambda,L)+\bm{w}$.
\end{proof}

We conclude this section by characterizing some interpolation spaces between $\mathscr{X}$ and $D(L)$. For this purpose, we introduce the following definition.

\begin{definition}
For every $\alpha\in (0,1)$, we denote by
$\mathscr{X}_{\alpha}$, the set of all functions $\bm{w}\in\mathscr{X}$ such that the functions $\dfrac{w_1}{q_L}$ and
$\dfrac{w_2}{q_H}$ are $\alpha$-H\"older continuous in $(-\infty,\eta^*]$ and in $[\eta^*,+\infty)$, respectively.
It is normed, by setting
\begin{eqnarray*}
\|\bm{w}\|_{\mathscr{X}_{\alpha}}=\left\|\frac{w_1}{q_L}\right\|_{C^{\alpha}_b((-\infty,\eta^*])}+\left\|\frac{w_2}{q_H}\right\|_{C^{\alpha}_b([\eta^*,+\infty))}
\end{eqnarray*}
for every $\bm{w}\in\mathscr{X}_{\alpha}$.

We also introduce the space $\mathscr{X}_{2+\alpha}$, defined as the set of all functions $\bm{w}\in\mathscr{X}_2$ such that
$D_{yy}\bm{w}=(D_{yy}w_1,D_{yy}w_2)\in\mathscr{X}_{\alpha}$. It is normed by setting
\begin{eqnarray*}
\|\bm{w}\|_{\mathscr{X}_{2+\alpha}}=
\|\bm{w}\|_{\mathscr{X}_2}+\left [\frac{D_{yy}w_1}{q_L}\right ]_{C^{\alpha}_b((-\infty,\eta^*])}+\left [\frac{D_{yy}w_2}{q_H}\right ]_{C^{\alpha}_b([\eta^*,+\infty))}
\end{eqnarray*}
for every $\bm{w}\in \mathscr{X}_{2+\alpha}$.
\end{definition}

\begin{theorem}
\label{thm-7.7}
For every $\alpha\in (0,1)$ it holds that
\begin{eqnarray*}
D_L(\alpha/2,\infty)=\{\bm{w}\in\mathscr{X}_{\alpha}: B_0\bm{w}=0\},\qquad\, D_L(1+\alpha/2,\infty)=\{\bm{w}\in\mathscr{X}_2:
{\mathcal B}\bm{w}=\bm{0}, B_0L\bm{w}=0\},
\end{eqnarray*}
with equivalence of the corresponding norms. Moreover, $\{\bm{w}\in\mathscr{X}_1: B_0\bm{w}=0\}\hookrightarrow D_L(1/2,\infty)$.
\end{theorem}

\begin{proof}
Throughout the proof, $C$ denotes a positive constant, independent of $t$, $\lambda$ and the functions that we will consider, which may vary from line to line.

Let us fix a function $\bm{w}\in{\mathscr X}_{\alpha}$ (here, $\alpha\in (0,1]$). Since $w_1(\eta^*)=w_2(\eta^*)$, the function $\zeta:\R\to\R$, defined by
\begin{eqnarray*}
\zeta(y)=
\left\{
\begin{array}{ll}
w_1(y), & y\le\eta^*,\\
w_2(y), & y>\eta^*
\end{array}
\right.
\end{eqnarray*}
is $\alpha$-H\"older continuous in $\R$. Moreover, its $C^{\alpha}_b(\R)$-norm coincides with the ${\mathscr X}_{\alpha}$-norm of $\bm{w}$.

We regularize function $\zeta$ by convolution. For this purpose, let $\varphi\in C^{\infty}_c(\R)$ be a nonnegative function with $L^1$-norm over $\R$ equal to one. For every $t>0$, we set
\begin{align*}
\zeta_t(x)=t\int_{\R}\varphi(ty)\zeta(x-y)dy,\qquad x\in\R.
\end{align*}

It is an easy task to check that
\begin{align}
&\|\zeta_t-\zeta\|_{\infty}\le Ct^{-\alpha}\|\bm{w}\|_{{\mathscr X}_{\alpha}},
\label{estim-1}\\
&\|D_x\zeta_t\|_{\infty}\le Ct^{1-\alpha}\|\bm{w}\|_{{\mathscr X}_{\alpha}},
\label{estim-2}
\\
&\|D_{xx}\zeta_t\|_{\infty}\le Ct^{2-\alpha}\|\bm{w}\|_{{\mathscr X}_{\alpha}}
\label{estim-3}
\end{align}
for every $t>0$.

Now, we are ready to prove that $\bm{w}\in D_L(\alpha/2,\infty)$. For this purpose, we recall that $D_L(\alpha/2,\infty)$ can be characterized as the set of all functions $\bm{w}\in{\mathscr X}$ for which there exists $\lambda_0>0$ such that
\begin{align*}
[\bm{w}]_{D_L(\alpha/2,\infty)}:=\sup_{\lambda\ge\lambda_0}\lambda^{\frac{\alpha}{2}}\|LR(\lambda,L)\bm{w}\|_{\mathscr{X}}<+\infty
\end{align*}
and the norm of $D_L(\alpha/2,\infty)$ is equivalent to the norm $\|\cdot\|_{\mathscr{X}}+[\,\cdot\,]_{D_L(\alpha/2,\infty)}$.

We introduce the function $\bm{\zeta}_{\sqrt{\lambda}}$, whose components are the restrictions of $\zeta_{\sqrt{\lambda}}$ to $(-\infty,\eta^*]$ and to $[\eta^*,+\infty)$, respectively. Next, we split,
\begin{align*}
LR(\lambda,L)\bm{w}=LR(\lambda,L)(\bm{w}-\bm{\zeta}_{\sqrt{\lambda}})
+LR(\lambda,L)\bm{\zeta}_{\sqrt{\lambda}},
\end{align*}
so that
\begin{align}
\|LR(\lambda,L)\bm{w}\|_{\mathscr X}\le &\|LR(\lambda,L)(\bm{w}-\bm{\zeta}_{\sqrt{\lambda}})\|_{\mathscr X}+
+\|LR(\lambda,L)\bm{\zeta}_{\sqrt{\lambda}}\|_{\mathscr{X}}\notag\\
\le & C\lambda^{-\frac{\alpha}{2}}\|LR(\lambda,L)\|_{L(\mathscr X)}\|\bm{w}\|_{\mathscr{X}_{\alpha}}+
\|LR(\lambda,L)\bm{\zeta}_{\sqrt{\lambda}}\|_{\mathscr{X}}.
\label{tg}
\end{align}

To estimate the second term in the last side of the previous inequality, we rewrite the function $\bm{u}_{\lambda}:=LR(\lambda,L)\bm{\zeta}_{\sqrt{\lambda}}$ in the following form:
\begin{align*}
\bm{u}_{\lambda}=(L-\lambda+\lambda)R(\lambda,L)\bm{\zeta}_{\sqrt{\lambda}}=-\bm{\zeta}_{\sqrt{\lambda}}+\lambda R(\lambda,L)\bm{\zeta}_{\sqrt{\lambda}},
\end{align*}
so that, applying operator ${\mathcal L}$ to the first and the last side of the previous formula, we deduce that
 $\bm{u}_{\lambda}$ is a solution to the equation $\lambda\bm{u}_{\lambda}-{\mathcal L}\bm{u}_{\lambda}=-{\mathcal L}\bm{\zeta}_{\sqrt{\lambda}}$.
Moreover,
\begin{align*}
{\mathcal B}\bm{u}_{\lambda}=-{\mathcal B}\bm{\zeta}_{\sqrt{\lambda}}
+\lambda {\mathcal B}R(\lambda,L)\bm{\zeta}_{\sqrt{\lambda}}
=-{\mathcal B}\bm{\zeta}_{\sqrt{\lambda}}.
 \end{align*}

Using \eqref{estim-2} and \eqref{estim-3}, we can easily infer that
\begin{align*}
\|{\mathcal L}\bm{\zeta}_{\sqrt{\lambda}}\|_{\mathscr{X}}\le C\lambda^{1-\frac{\alpha}{2}}\|\bm{w}\|_{\mathscr{X}_{\alpha}}
\end{align*}
and
\begin{align*}
\|{\mathcal B}\bm{\zeta}_{\sqrt{\lambda}}\|_{\mathscr{X}}\le C\lambda^{\frac{1}{2}-\frac{\alpha}{2}}\|\bm{w}\|_{\mathscr{X}_{\alpha}}
\end{align*}
for every $\lambda\ge 1$.

Now, applying estimate \eqref{estimate-boundary}, we deduce that, if $\lambda$ is sufficiently large, then
\begin{align*}
\lambda\|\bm{u}_{\lambda}\|_{\mathscr{X}}\le C(\|{\mathcal L}\bm{w}_{\sqrt{\lambda}}\|_{\mathscr{X}}+\sqrt{|\lambda|}\|{\mathcal B}\bm{\zeta}_{\sqrt{\lambda}}\|_{\mathscr{X}})\le C|\lambda|^{1-\frac{\alpha}{2}}\|\bm{\zeta}\|_{{\mathscr X}_{\alpha}}
\end{align*}
so that
\begin{equation*}
\|LR(\lambda,L)\bm{\zeta}_{\sqrt{\lambda}}\|_{\mathscr X}
=\|\bm{u}_{\lambda}\|_{\mathscr X}\le C|\lambda|^{-\frac{\alpha}{2}}\|\bm{w}\|_{\mathscr{X}_{\alpha}}
\end{equation*}

This estimate and \eqref{tg} show that there exists $\lambda_0>0$ such that
\begin{eqnarray*}
\lambda^{\frac{\alpha}{2}}\|LR(\lambda,L)\bm{w}_{\sqrt{\lambda}}\|_{\mathscr X}
\le C\|\bm{w}\|_{{\mathscr X}_{\alpha}}
\end{eqnarray*}
for every $\lambda\ge\lambda_0$. The embedding $\{\bm{w}\in {\mathscr X}_{\alpha}: B_0\bm{w}=0\}\hookrightarrow D_L(\alpha,\infty)$ follows for every $\alpha\in (0,1]$.

Let us now fix $\alpha\in (0,1)$ and prove that $D_L(\alpha/2,\infty)$ is continuously embedded into
$\{\bm{w}\in {\mathcal X}_{\alpha}: B_0\bm{w}=0\}$. For this purpose, we recall that $D_L(\alpha/2,\infty)=(\mathscr{X},D(L))_{\alpha/2,\infty}$ with equivalence of the corresponding norms. Since
$D(L)$ is continuously embedded into ${\mathscr X}_2$,
it follows that $D_L(\alpha/2,\infty)\hookrightarrow (\mathscr{X},\mathscr{X}_2)_{\alpha/2,\infty}$. To characterize this latter space, we take advantage of Lemma
\ref{lemma-char-eq} and introduce the isomorphism $T:\mathscr{X}\to C_b((-\infty,\eta^*];\C)\times C_b([\eta^*,+\infty);\C)$,
defined by $T\bm{w}=\left (\frac{w_1}{q_L},\frac{w_2}{q_H}\right )$ for every $\bm{w}\in {\mathscr X}$,  which is also bijective
between $\mathscr{X}_2$ and $C_b^2((-\infty,\eta^*];\C)\times C_b([\eta^*,+\infty);\C)$. The theory of interpolation spaces, it follows that $T$ is also bijective between
$(\mathscr{X},\mathscr{X}_2)_{\alpha,\infty}$ and $(C_b((-\infty,\eta^*];\C)\times C_b([\eta^*,+\infty);\C),C_b^2((-\infty,\eta^*];\C)\times C_b^2([\eta^*,+\infty);\C))_{\alpha/2,\infty}=
(C_b((-\infty,\eta^*];\C);C_b^2((-\infty,\eta^*];\C))_{\alpha/2,\infty}\times (C_b([\eta^*,+\infty);\C),C_b^2([\eta^*,+\infty);\C))_{\alpha/2,\infty}$, with equivalence of the corresponding norms.
Since this latter space coincides with $C_b^{\alpha}((-\infty,\eta^*];\C)\times C_b^{\alpha}([\eta^*,+\infty);\C)$, i.e., with the space $\mathscr{X}_{\alpha}$, with equivalence of the corresponding norms. We have so proved that $D_L(\alpha/2,+\infty)$ is continuously embedded into $\mathscr{X}_{\alpha}$. To complete the proof of the inclusion
$D_L(\alpha/2,+\infty)\hookrightarrow \{\bm{w}\in\mathscr{X}_{\alpha}: B_0\bm{w}\}$, we recall that $D(L)$ is dense in $D_L(\alpha/2,+\infty)$, with respect to the norm of this latter space. Thus, we can determine a sequence $\{\bm{v}_n\}\subset D(L)$ such that $\|\bm{v}_n-\bm{w}\|_{D_L(\alpha/2,+\infty)}$ converges to $0$ as $n$ tends to $+\infty$. In particular, $\|\bm{v}_n-\bm{w}\|_{\mathscr{X}}$ vanishes as $n$ tends to $+\infty$. Since $B_0\bm{v}_n=0$ for every $n\in\N$, letting $n$ tend to $+\infty$, we get the condition $B_0\bm{w}=0$.

The topological equivalence $D_L(\alpha/2,+\infty)=\{\bm{w}\in\mathcal{X}_{\alpha}: B_0\bm{w}=0\}$ is now checked.

\medskip

Let us now prove the second characterization, For this purpose, we begin by recalling that $D_L(1+\alpha/2,+\infty)$ is the set of all $\bm{w}\in D(L)$ such that $L\bm{w}\in \mathscr{X}_{\alpha}$. From Theorem \ref{thm:gen} it follows that $D(L)\subset\mathscr{X}_2$ and
\begin{equation}
\|\bm{w}\|_{D(L)}\le C_1\|\bm{w}\|_{\mathscr{X}_2},\qquad\;\, \bm{w}\in D(L).
\label{cane}
\end{equation}
On the other hand, since
\begin{eqnarray*}
L\bm{w}=\left (
\frac{1}{2}\sgL2 \frac{\partial^2 w_1}{\partial y^2} + \left(\delta -\frac{1}{2}\sgL2\right)\frac{\partial w_1}{\partial y},\frac{1}{2}\sgH2 \frac{\partial^2 w_2}{\partial y^2} + \left(\delta -
\frac{1}{2}\sgH2\right)\frac{\partial w_2}{\partial y}\right ),\qquad\;\,\bm{w}\in D(L),
\end{eqnarray*}
it follows that
\begin{equation}
\frac{\partial^2w_j}{\partial y^2}=\frac{2}{\sigma^2_L}\bigg ((L\bm{w})_j-\bigg (\delta-\frac{1}{2}\sigma^2_L\bigg )\frac{\partial w_j}{\partial y}\bigg ),\qquad\;\,j=1,2,
\label{Dyy-L}
\end{equation}
so that
\begin{equation}
\bigg\|\frac{D_{yy}w_1}{q_L}\bigg\|_{\infty}\le
C\left (\|L\bm{w}\|_{\mathscr{X}}+\bigg\|\frac{D_yw_1}{q_L}\bigg\|_{\infty}\right ).
\label{norm-equiv}
\end{equation}

Further, we recall that the space $C^1_b((-\infty,\eta^*])$ belongs to the class $J_{1/2}$ between $C_b((-\infty,\eta^*])$ and
$C^2_b((-\infty,\eta^*])$, which means that
\begin{eqnarray*}
\|f\|_{C^1_b((-\infty,\eta^*])}\le C\|f\|_{\infty}^{\frac{1}{2}}\|f\|_{C^2_b((-\infty,\eta^*])}^{\frac{1}{2}},\qquad\;\,f\in C^2_b((-\infty,\eta^*]).
\end{eqnarray*}

Applying this estimate with $f=\dfrac{w_1}{q_L}$, it follows that
\begin{align*}
\left\|\dfrac{D_yw_1}{q_L}\right\|_{\infty}-C
\left\|\dfrac{w_1}{q_L}\right\|_{\infty}\le
C\left\|\dfrac{w_1}{q_L}\right\|_{\infty}^{\frac{1}{2}}
\left (\left\|\dfrac{w_1}{q_L}\right\|_{\infty}+
\left\|\dfrac{D_yw_1}{q_L}\right\|_{\infty}+
\left\|\dfrac{D_{yy}w_1}{q_L}\right\|_{\infty}\right )^{\frac{1}{2}}.
\end{align*}
Therefore, using Cauchy-Schwarz inequality, from the previous inequality we deduce that
\begin{align*}
\left\|\dfrac{D_yw_1}{q_L}\right\|_{\infty}-C
\left\|\dfrac{w_1}{q_L}\right\|_{\infty}
\le \varepsilon \left\|\dfrac{D_{yy}w_1}{q_L}\right\|_{\infty}
+K_{\varepsilon}\left\|\dfrac{w_1}{q_L}\right\|_{\infty}
+C\varepsilon\left\|\dfrac{D_{y}w_1}{q_L}\right\|_{\infty}
\end{align*}
for every $\varepsilon>0$ and some positive constant $K_{\varepsilon}$, independent of $\bm{w}$ and blowing up as $\varepsilon$ tends to $0$. Taking $\varepsilon$ sufficiently small, we conclude that
\begin{align}
\left\|\dfrac{D_yw_1}{q_L}\right\|_{\infty}
\le \varepsilon \left\|\dfrac{D_{yy}w_1}{q_L}\right\|_{\infty}
+K_{\varepsilon}'\left\|\dfrac{w_1}{q_L}\right\|_{\infty}
\label{stima-J12}
\end{align}
for some positive constant $K_{\varepsilon}'$, independent of $w_1$. Replacing this inequality into \eqref{norm-equiv}, gives
\begin{eqnarray*}
\bigg\|\frac{D_{yy}w_1}{q_L}\bigg\|_{\infty}\le C\|L\bm{w}\|_{\mathscr{X}}+C\varepsilon \bigg\|\frac{D_{yy}w_1}{q_L}\bigg\|_{\infty}+CK_{\varepsilon}'\bigg\|\frac{w_1}{q_L}\bigg\|_{\infty}.
\end{eqnarray*}

Taking $\varepsilon$ smaller, if needed, we can get rid of the norm of $D_{yy}w_1$ from the right-hand side of the previous estimate and conclude that
\begin{equation}
\bigg\|\frac{D_{yy}w_1}{q_L}\bigg\|_{\infty}\le C\bigg (\|L\bm{w}\|_{\mathscr{X}}+\bigg\|\frac{w_1}{q_L}\bigg\|_{\infty}\bigg )
\le C\|\bm{w}\|_{D(L)}.
\label{stima-equiv-1}
\end{equation}

From \eqref{stima-J12} and \eqref{stima-equiv-1}, it follows that
\begin{align*}
\bigg\|\frac{w_1}{q_L}\bigg\|_{\infty}+\bigg\|\frac{D_{y}w_1}{q_L}\bigg\|_{\infty}+\bigg\|\frac{D_{yy}w_1}{q_L}\bigg\|_{\infty}
\le C\|\bm{w}\|_{D(L)}.
\end{align*}

Arguing in the same way, we can deal with the function $w_2$ and we conclude that
\begin{equation}
\|\bm{w}\|_{\mathscr{X}_2}\le C\|\bm{w}\|_{D(L)},\qquad\;\,\bm{w}\in D(L).
\label{salvato}
\end{equation}

Estimates \eqref{cane} and \eqref{salvato} show that the graph norm of $L$ is equivalent to the norm of $\mathscr{X}_2$ on $D(L)$.

We can now complete the characterization of $D_L(1+\alpha/2,+\infty)$. For this purpose, we fix $\bm{w}\in D_L(1+\alpha/2,+\infty)$. Then, $L\bm{w}\in\mathscr{X}_{\alpha}$, which in view of \eqref{Dyy-L}, implies that $D_{yy}\bm{w}\in \mathscr{X}_{\alpha}$, so that $\bm{w}\in\mathscr{X}_{2+\alpha}$. Moreover,
\begin{align*}
\left [\frac{D_{yy}w_1}{q_L}\right ]_{C^{\alpha}{(-\infty,\eta^*])}}
\le &\frac{2}{\sigma_L^2}\left (\|L\bm{w}\|_{\mathscr{X}_{\alpha}}
+\left |\delta-\frac{1}{2}\sigma_L^2\right |\left [\frac{D_yw_1}{q_L}\right ]_{C^{\alpha}((-\infty,\eta^*])}\right )\\
\le &C\left (\|L\bm{w}\|_{\mathscr{X}_{\alpha}}++\left |\delta-\frac{1}{2}\sigma_L^2\right |\left\|\frac{w_1}{q_L}\right\|_{C^1_b((-\infty,\eta^*])}\right )\\
\le &C(\|L\bm{w}\|_{\mathscr{X}_{\alpha}}+
\|\bm{w}\|_{\mathscr{X}_2})\\
\le &C\|\bm{w}\|_{D_L(1+\alpha/2,+\infty)}
\end{align*}
for every $\bm{w}\in D_L(1+\alpha/2,+\infty)$.
Clearly, $\mathcal{B}\bm{w}=\bm{0}$ since $\bm{w}\in D(L)$, Moreover, $B_0L\bm{w}=0$ since $L\bm{w}\in D_L(\alpha/2,+\infty)$.
We have so proved the topological inclusion $D_L(1+\alpha/2,+\infty)\hookrightarrow \mathscr{X}_{2+\alpha}$.

Viceversa, let us suppose that $\bm{w}\in\mathscr{X}_{2+\alpha}$ is such that $\mathcal{B}\bm{w}=0$ and $B_0L\bm{w}=0$. Then, $\bm{w}\in D(L)$ and $L\bm{w}\in\mathcal{X}_{2+\alpha}$. Since $B_0L\bm{w}=0$, from the characterization of the space $D_L(\alpha/2,+\infty)$, it follows that $L\bm{w}\in D_L(\alpha/2,+\infty)$. Therefore,
$\bm{w}\in D_L(1+\alpha/2,+\infty)$. Moreover, $\|L\bm{w}\|_{\mathscr{X}_{\alpha}}\le C\|\bm{w}\|_{\mathscr{X}_{2+\alpha}}$. From this estimate and the previous results, we conclude that $\mathscr{X}_{2+\alpha}\hookrightarrow D_L(1+\alpha/2,+\infty)$. The proof is now complete.
\end{proof}

\section{A thorough analysis of the dispersion relation}\label{dispersion_relation}
In this subsection, we analyze the space the solution to the dispersione relation ${\mathcal D}_{\lambda}$ in order to have a complete characterization of the spectrum of the operator $L$ and its localization with respect to the imaginary axis.

\begin{theorem}
The solutions of the dispersion relation ${\mathcal D}_{\lambda}=0$ are contained in the left-halfplane $\{\lambda\in\C: {\rm Re}\lambda>0\}$ if and only if
\begin{eqnarray*}
\left\{
\begin{array}{l}
(c_L-2)(c_H-c_L)\ge 0,\\[1mm]
(c_L-1)(2c_H^2+c_Hc_L^2-4c_Hc_L+c_L^2)\le 0,\\[1mm]
(c_L-c_H)(-\sigma_H^2c_Hc_L^2+2\sigma^2_Hc_Hc_L-2\sigma_H^2c_H-\sigma_H^2c_L^2+2\sigma_H^2c_L+2\sigma_L^2c_Hc_L-2\sigma_L^2c_H\\
\qquad\qquad\quad +\sigma_L^2c_L^3-3\sigma_L^2c_L^2+2\sigma_L^2c_L)\ge 0
\end{array}
\right.
\end{eqnarray*}
or
\begin{eqnarray*}
\left\{
\begin{array}{l}
(c_L-2)(c_H-c_L)\ge 0,\\[1mm]
(c_L-1)(2c_H^2+c_Hc_L^2-4c_Hc_L+c_L^2)\textcolor{red}{\le}0,\\[1mm]
(c_L-c_H)(-\sigma_H^2c_Hc_L^2+2\sigma^2_Hc_Hc_L-2\sigma_H^2c_H-\sigma_H^2c_L^2+2\sigma_H^2c_L+2\sigma_L^2c_Hc_L-2\sigma_L^2c_H\\
\qquad\qquad\quad +\sigma_L^2c_L^3-3\sigma_L^2c_L^2+2\sigma_L^2c_L)\textcolor{red}{<} 0\\[1mm]
\displaystyle c_H(c_L-1)^2(c_H-c_L)^2(c_H+c_L^2-2c_L)\textcolor{red}{\le} 0,\\[1mm]
\Delta>0,
\end{array}
\right.
\end{eqnarray*}
where $\Delta$ is defined in \eqref{delta}.
\end{theorem}

\begin{proof}
We recall that the function $\lambda\mapsto {\mathcal D}_{\lambda}$
is defined by
\begin{align*}
{\mathcal D}_{\lambda}=\frac{c_H}{c_L}+\frac{1}{2}(c_L-c_H)-1
+\sqrt{\frac{1}{4}(c_H-1)^2+\frac{2\lambda}{\sigma^2_H}}
+\sqrt{\frac{1}{4}(c_L-1)^2+\frac{2\lambda}{\sigma^2_L}}
\end{align*}
for every $\lambda\in\C$. The equation ${\mathcal D}_{\lambda}=0$ is easily seen to be equivalent to
\begin{align*}
\sqrt{\frac{1}{4}(c_H-1)^2+\frac{2\lambda}{\sigma^2_H}}
+\sqrt{\frac{1}{4}(c_L-1)^2+\frac{2\lambda}{\sigma^2_L}}
=1-\frac{c_H}{c_L}-\frac{1}{2}(c_L-c_H)
\end{align*}
Clearly, if $1-\frac{c_H}{c_L}-\frac{1}{2}(c_L-c_H)<0$ then the previous equation does not admit solutions in $\mathbb C$ since the left-hand side of the above equation has nonnegative real part (we recall that by $\sqrt{\cdot}$ we denote the principal square root).

Hence, from now on, we assume that 
\begin{equation}
1-\frac{c_H}{c_L}-\frac{1}{2}(c_L-c_H)\ge 0,\qquad\;\,{\rm i.e.,}\qquad\;\,(c_L-2)(c_H-c_L)\ge 0.
\label{checco}
\end{equation}
Then, squaring the above equation, we find the equivalent equation
\begin{align*}
&2\sqrt{\frac{1}{4}(c_H-1)^2+\frac{2\lambda}{\sigma^2_H}}\sqrt{\frac{1}{4}(c_L-1)^2+\frac{2\lambda}{\sigma^2_L}}
\\
=&\frac{c_H^2}{c_L^2}-\frac{1}{2}c_Hc_L+\frac{1}{2}
-2\frac{c_H}{c_L}+\frac{5}{2}c_H-\frac{c_H^2}{c_L}-\frac{1}{2}c_L-2\lambda\bigg (\frac{1}{\sigma^2_H}+\frac{1}{\sigma^2_L}\bigg ).
\end{align*}
Clearly, the previous equation cannot admit solutions $\lambda$ with
\begin{align}
{\rm Re}(\lambda)>c_{H,L}:=&\frac{\sigma^2_H\sigma^2_L}{2(\sigma_H^2+\sigma_L^2)}\bigg (\frac{c_H^2}{c_L^2}-\frac{1}{2}c_Hc_L+\frac{1}{2}
-2\frac{c_H}{c_L}+\frac{5}{2}c_H-\frac{c_H^2}{c_L}-\frac{1}{2}c_L\bigg )\notag\\
=&-\frac{\sigma^2_H\sigma^2_L(c_L-1)(2c_H^2+c_Hc_L^2-4c_Hc_L+c_L^2)}{4(\sigma_H^2+\sigma_L^2)c_L^2}.
\label{zalone}
\end{align}

In particular, from \eqref{zalone} it follows that the dispersion relation ${\mathcal D}_{\lambda}=0$ cannot admit solution with nonnegative real parts if $(c_L-1)(2c_H^2+c_Hc_L^2-4c_Hc_L+c_L^2)>0$. Hence, from now on we assume that 
\begin{equation}
(c_L-1)(2c_H^2+c_Hc_L^2-4c_Hc_L+c_L^2)\le 0.
\label{unita}
\end{equation}

On the other hand, if ${\rm Re}(\lambda)\le c_{H,L}$, then we can square once more the above equation and get
\begin{align*}
&\frac{1}{4}c_H^2c_L^2+\frac{1}{4}c_H^2-\frac{1}{2}c_H^2c_L+\frac{1}{4}c_L^2+\frac{1}{4}-\frac{1}{2}c_L-\frac{1}{2}c_Hc_L^2-\frac{1}{2}c_H+c_Hc_L\\
&+2\lambda\bigg (\frac{1}{\sigma^2_L}(c_H^2+1-2c_H)+\frac{1}{\sigma_H^2}(c_L^2+1-2c_L)\bigg )+\frac{16}{\sigma^2\sigma^2_L}\lambda^2\\
=&\bigg (c_H^2+1-2c_H+\frac{8\lambda}{\sigma^2_H}\bigg )\bigg (\frac{1}{4}c_L^2+\frac{1}{4}-\frac{1}{2}c_L+\frac{2\lambda}{\sigma^2_L}\bigg )\\
=&\frac{c_H^4}{c_L^4}+\frac{1}{4}c_H^2c_L^2+\frac{1}{4}+5\frac{c_H^2}{c_L^2}+\frac{37}{4}c_H^2+\frac{c_H^4}{c_L^2}+\frac{1}{4}c_L^2+4\lambda^2\bigg (\frac{1}{\sigma^2_H}+\frac{1}{\sigma^2_L}\bigg )^2-6\frac{c^3_H}{c_L}-4\frac{c_H^3}{c_L^3}\\
&+9\frac{c_H^3}{c_L^2}-2\frac{c_H^4}{c_L^3}-4\lambda\bigg (\frac{1}{\sigma^2_H}+\frac{1}{\sigma^2_L}\bigg )\bigg (\frac{c_H^2}{c_L^2}-\frac{1}{2}c_Hc_L+\frac{1}{2}-2\frac{c_H}{c_L}+\frac{5}{2}c_H-\frac{c_H^2}{c_L}-\frac{1}{2}c_L\bigg )\\
&-3c_Hc_L-\frac{5}{2}c_H^2c_L+c_H^3+\frac{1}{2}c_Hc_L^2-2\frac{c_H}{c_L}+\frac{9}{2}c_H-\frac{1}{2}c_L-12\frac{c_H^2}{c_L}
\end{align*}
i.e.,
\begin{align}
&4\lambda^2\bigg (\frac{1}{\sigma^2_H}-\frac{1}{\sigma^2_L}\bigg )^2-2\lambda\bigg [\frac{1}{\sigma^2_L}\bigg (c_H^2+2+3c_H+2\frac{c_H^2}{c_L^2}-c_Hc_L-4\frac{c_H}{c_L}-2\frac{c_H^2}{c_L}-c_L\bigg )\notag\\
&\phantom{4\lambda^2\bigg (\frac{1}{\sigma^2_H}-\frac{1}{\sigma^2_L}\bigg )^2+2\lambda\bigg [\;}+\frac{1}{\sigma_H^2}\bigg (c_L^2+2-3c_L+2\frac{c_H^2}{c_L^2}-c_Hc_L-4\frac{c_H}{c_L}+5c_H-2\frac{c_H^2}{c_L}\bigg )\bigg ]\notag\\
=&-\frac{c_H^4}{c_L^4}-5\frac{c_H^2}{c_L^2}-9c_H^2-\frac{c_H^4}{c_L^2}+6\frac{c^3_H}{c_L}+4\frac{c_H^3}{c_L^3}-9\frac{c_H^3}{c_L^2}+2\frac{c_H^4}{c_L^3}+12\frac{c_H^2}{c_L}+4c_Hc_L+2c_H^2c_L-c_H^3\notag\\
&-c_Hc_L^2+2\frac{c_H}{c_L}-5c_H.
\label{disp-relat}
\end{align}

Summing up, the dispersion relation ${\mathcal D}_{\lambda}=0$ is equivalent to the system
\begin{eqnarray*}
\left\{
\begin{array}{l}
\textrm{eqn. } \eqref{disp-relat}\\[1mm]
{\rm Re}(\lambda)\le c_{H,L}
\end{array}
\right.
\end{eqnarray*}
provided that conditions \eqref{checco} and \eqref{unita} are both satisfied.

We turn back to the equation \eqref{disp-relat}
and we begin by observing that the term in square brackets can be factorized as follows
\begin{align}
\frac{(c_L-c_H)}{\sigma_H^2\sigma_L^2c_L^2}\big (&-\sigma_H^2c_Hc_L^2+2\sigma_H^2c_Hc_L
-2\sigma_H^2c_H-\sigma_H^2c_L^2+2\sigma_H^2c_L
+2\sigma_L^2c_Hc_L-2\sigma_L^2c_H\notag\\
&+\sigma_L^2c_L^3-3\sigma_L^2c_L^2+2\sigma_L^2c_L\big ).
\label{venditore}
\end{align}
Clearly, if \eqref{venditore} is nonnegative, then \eqref{disp-relat} admits a solution with nonnegative part.

Now, we suppose now that \eqref{venditore} is negative and  compute the discriminant of \eqref{disp-relat}
\begin{align}
\Delta=&
\bigg [\frac{1}{\sigma^2_L}\bigg (c_H^2+2+3c_H+2\frac{c_H^2}{c_L^2}-c_Hc_L-4\frac{c_H}{c_L}-2\frac{c_H^2}{c_L}-c_L\bigg )\notag\\
&\;\;
+\frac{1}{\sigma_H^2}\bigg (c_L^2+2-3c_L+2\frac{c_H^2}{c_L^2}-c_Hc_L-4\frac{c_H}{c_L}+5c_H-2\frac{c_H^2}{c_L}\bigg )\bigg ]^2\notag\\
&-4\bigg (\frac{1}{\sigma^2_H}-\frac{1}{\sigma^2_L}\bigg )^2
\bigg [-\frac{c_H^4}{c_L^4}-5\frac{c_H^2}{c_L^2}-9c_H^2-\frac{c_H^4}{c_L^2}+6\frac{c^3_H}{c_L}+4\frac{c_H^3}{c_L^3}-9\frac{c_H^3}{c_L^2}+2\frac{c_H^4}{c_L^3}+12\frac{c_H^2}{c_L}\notag\\
&\phantom{-4\bigg (\frac{1}{\sigma^2_H}-\frac{1}{\sigma^2_L}\bigg )^2\;\;\;\,} +4c_Hc_L+2c_H^2c_L-c_H^3-c_Hc_L^2+2\frac{c_H}{c_L}-5c_H\bigg ]\notag\\
=&\frac{(c_H-c_L)^2}{c_L^4}\bigg [\frac{1}{\sigma_H^4}(c_L-1)^2(8c_H^2+8c_Hc_L^2-16c_Hc_L+c_L^4-4c_L^3+4c_L^2)\notag\\
&\qquad\qquad\qquad+\frac{1}{\sigma^4_L}(c_H^2c_L^4-4c_H^2c_L^3+12c_H^2c_L^2-16c_H^2c_L+8c_H^2+6c_Hc_L^4-24c_Hc_L^3\notag\\
&\qquad\qquad\qquad\qquad\quad+32c_Hc_L^2-16c_Hc_L+c_L^4-4c_L^3+4c_L^2)\notag\\
&\qquad\qquad\qquad-\frac{2}{\sigma_H^2 \sigma_L^2}(c_L-1)(2c_H^2c_L^2+c_Hc_L^4-4c_Hc_L^2+c_L^4-4c_L^3+4c_L^2)\bigg ].
\label{delta}
\end{align}

Clearly, if $\Delta\ge 0$, then the equation \eqref{disp-relat} does not admit solutions with nonnegative real part. On the other hand, if $\Delta>0$, such an equation admits a real solution with nonnegative real part if and only if and only if the term on the right-hand side of \eqref{disp-relat} is nonnegative. Such a term can be factorized as follows:
\begin{eqnarray*}
-\frac{c_H(c_L-1)^2(c_H-c_L)^2(c_H+c_L^2-2c_L)}{c_L^4}    
\end{eqnarray*}

Summing up, we have proved that the condition to have roots of the dispersion relation with nonnegative real part are
\begin{eqnarray*}
\left\{
\begin{array}{l}
(c_L-2)(c_H-c_L)\ge 0,\\[1mm]
(c_L-1)(2c_H^2+c_Hc_L^2-4c_Hc_L+c_L^2)\le 0,\\[1mm]
(c_L-c_H)(-\sigma_H^2c_Hc_L^2+2\sigma^2_Hc_Hc_L-2\sigma_H^2c_H-\sigma_H^2c_L^2+2\sigma_H^2c_L+2\sigma_L^2c_Hc_L-2\sigma_L^2c_H\\
\qquad\qquad\quad +\sigma_L^2c_L^3-3\sigma_L^2c_L^2+2\sigma_L^2c_L)\ge 0
\end{array}
\right.
\end{eqnarray*}
or
\begin{eqnarray*}
\left\{
\begin{array}{l}
(c_L-2)(c_H-c_L)\ge 0,\\[1mm]
(c_L-1)(2c_H^2+c_Hc_L^2-4c_Hc_L+c_L^2)\textcolor{red}{\le} 0,\\[1mm]
(c_L-c_H)(-\sigma_H^2c_Hc_L^2+2\sigma^2_Hc_Hc_L-2\sigma_H^2c_H-\sigma_H^2c_L^2+2\sigma_H^2c_L+2\sigma_L^2c_Hc_L-2\sigma_L^2c_H\\
\qquad\qquad\quad +\sigma_L^2c_L^3-3\sigma_L^2c_L^2+2\sigma_L^2c_L)\textcolor{red}{<} 0\\[1mm]
\displaystyle c_H(c_L-1)^2(c_H-c_L)^2(c_H+c_L^2-2c_L)\textcolor{red}{\le} 0,\\[1mm]
\Delta>0.
\end{array}
\right.
\end{eqnarray*}
The proof is complete.
\end{proof}

\begin{corollary}
In the domain $(\AA)$ of financially admissible parameters, the spectrum of the operator $L$
consists of the halfline $\left (-\infty,-\min\left\{\frac{c_L^2}{8}(c_L^2-1)^2,\frac{c_H^2}{8}(c_H^2-1)^2\right\}\right ]$.
\end{corollary}

\begin{proof}
By condition \eqref{admissible_bis}, it holds that $c_L<1<c_H$, so that $(c_L-2)(c_H-c_L)<0$. Hence, the dispersion relation has no solutions and the statement follows from Theorem \ref{thm:gen}.
\end{proof}

\section{Nonlinear stability analysis}\label{nonlinear_stability}

In this section, we  analyze the nonlinear Cauchy problem \eqref{FNLP}.

For further use, we introduce some notation. For every $T\in [0,+\infty)\times\{+\infty\}$, by ${\mathcal I}_T$ we denote the interval $[0,T]$, if $T\in\R$, and the halfline $[0,+\infty)$ otherwise.
Moreover, $I_-$ and $I_+$ denote, respectively, the intervals $(-\infty,\eta^*]$ and $[\eta^*,+\infty)$. Finally, for every $\omega>0$ and every pair $\bm{v}=(v_1,v_2)$, with $v_1:{\mathcal I}_T\times I_-\to\R$
 and $v_2:{\mathcal I}_T\times I_+\to\R$, we set $v_{1,\omega}^{\sharp}(t,x)=\displaystyle e^{\omega t}\frac{v_1(t,x)}{q_L(x)}$ and
 $v_{2,\omega}^{\sharp}(t,x)=\displaystyle e^{\omega t}\frac{w_1(t,x)}{q_H(x)}$ and, then $\bm{v}^{\sharp}_{\omega}=(v_{1,\omega}^{\sharp},v_{2,\omega}^{\sharp})$.

Let us introduce the following function spaces.

\begin{definition}
For every $T\in (0,+\infty)\cup\{+\infty\}$, every $\omega\ge 0$ and every $\alpha\in (0,1)$, the space
\begin{itemize}
\item
${\mathscr X}_{\alpha,0}(T,\omega)$ is the set of all pairs $\bm{v}=(v_1,v_2)$ such that
$v_{1,\omega}^{\sharp}\in C^{\alpha,0}_b({\mathcal I}_T\times I_-;\R)$ and $v_{2,\omega}^{\sharp}\in C^{\alpha,0}_b({\mathcal I}_T\times I_+;\R)$. It is endowed with the norm
\begin{align*}
\|\bm{v}\|_{{\mathscr X}_{\alpha,0}(T,\omega)}=\|v_{1,\omega}^{\sharp}\|_{C^{\alpha,0}_b({\mathcal I}_T\times I_-;\R)}+\|v_{2,\omega}^{\sharp}\|_{C^{\alpha,0}_b({\mathcal I}_T\times I_+;\R)};
\end{align*}
\item
${\mathscr X}_{\alpha/2,\alpha}(T,\omega)$ is the set of all pairs $\bm{v}=(v_1,v_2)$ such that
$v_{1,\omega}^{\sharp}\in C^{\alpha/2,\alpha}_b({\mathcal I}_T\times I_-;\R)$ and $v_{2,\omega}^{\sharp}\in C^{\alpha/2,\alpha}_b({\mathcal I}_T\times I_+;\R)$. It is endowed with the norm
\begin{align*}
\|\bm{v}\|_{{\mathscr X}_{\alpha/2,\alpha}(T,\omega)}=\|v_{1,\omega}^{\sharp}\|_{C^{\alpha/2,\alpha}_b({\mathcal I}_T\times I_-;\R)}+\|v_{2,\omega}^{\sharp}\|_{C^{\alpha/2,\alpha}_b({\mathcal I}_T\times I_+;\R)};
\end{align*}
\item
${\mathscr X}_{1+\alpha/2,2+\alpha}(T,\omega)$ is the set of all pairs $\bm{v}=(v_1,v_2)$ such that
$v_{1,\omega}^{\sharp}\in C^{1+\alpha/2,2+\alpha}_b({\mathcal I}_T\times I_-;\R)$ and $v_{2,\omega}^{\sharp}\in C^{1+\alpha/2,2+\alpha}_b({\mathcal I}_T\times I_+;\R)$. It is endowed with the norm
\begin{align*}
\|\bm{v}\|_{{\mathscr X}_{1+\alpha/2,2+\alpha}(T,\omega)}=\|v_{1,\omega}^{\sharp}\|_{C^{1+\alpha/2,2+\alpha}_b({\mathcal I}_T\times I_-;\R)}+\|v_{2,\omega}^{\sharp}\|_{C^{1+\alpha/2,2+\alpha}_b({\mathcal I}_T\times I_+;\R)}.
\end{align*}
\end{itemize}
\end{definition}

Sometimes, we find it useful to denote by $\|\bm{v}\|_{\mathscr{X}_{0,0}(T,\omega)}$ the sum 
$\|v_1^{\sharp}\|_{C_b({\mathcal I}_T\times I_-)}+\|v_2^{\sharp}\|_{C_b({\mathcal I}_T\times I_+)}$.

To begin with, we observe that the function ${\mathcal F}$ is well defined in the subset of $\mathscr{X}_2$ of pairs $(w_1,w_2)$ such that $1-\Psi'(w_1(\eta^*))\left\{\gamma e^{\eta^*}\left (1-\frac{2\delta}{\sigma^2_L}\right )^2\Psi(w_1(\eta^*))+D_yw_1(\eta^*_-)\right\}\neq 0$ and
\begin{align}\label{Fscript-1}
\FF(\bm{w})&=\frac{\frac{\sgL2}{2}\Psi'(w_1(\eta^*))D_{yy}w_1(\eta^*) + \left(\delta-
\frac{1}{2}\sgL2\right)\Psi'(w_1(\eta^*))D_yw_1(\eta^*)}
{1-\Psi'(w_1(\eta^*))\left\{\gamma e^{\eta^*}\left (1-\frac{2\delta}{\sigma^2_L}\right )^2\Psi(w_1(\eta^*))+D_yw_1(\eta^*_-)\right\}} \nonumber\\
&\qquad\;\, \times \bigg(D_y\bm{w}+\gamma e^{\eta^*}\bigg (1-\frac{2\delta}{\sgL2}\bigg )^2\Psi(\bm{w})\bigg),
\end{align}
for such functions $\bm{w}$. Since $\Psi(0)=0$, it follows immediately that ${\mathcal F}$ is well defined in $B(0,\rho)\subset D(L)$ for a sufficiently small $\rho$.
In the following lemma we prove some properties of this function and of function ${\mathcal G}$, defined by
\begin{eqnarray}\label{Gscript-1}
\GG(\bm{w})= R(w_1(\eta^*))\gamma e^{\eta^*}\left(1-\frac{2\delta}{\sgL2}\right) \frac{2\delta}{\sgH2\sgL2}(\sgL2-\sgH2),
\end{eqnarray}
for every $\bm{w}$ as above, where $R$ is a smooth function (see \eqref{invert_linear}), which are crucial in the stability analysis.

\begin{lemma}
There exists $\rho_0>0$ such that for every $\rho\in (0,\rho_0]$, every $T\in (0,+\infty)\cup\{\infty\}$ and $\omega\in (0,\infty)$, the following properties are satisfied.
\begin{enumerate}
\item[\rm (i)]
The function ${\mathcal F}$ maps the ball $B(0,\rho)$ of ${\mathscr X}_{1+\alpha/2,2+\alpha}(T,\omega)$ into ${\mathscr X}_{\alpha/2,\alpha}(T,\omega)$. Moreover, there exists a positive constant $C$ such that
\begin{equation}
\|{\mathcal F}(\bm{w}_2)-{\mathcal F}(\bm{w}_1)\|_{\mathscr{X}_{\alpha/2,\alpha}(T,\omega)}\le C\rho\|\bm{w}_2-\bm{w}_1\|_{\mathscr{X}_{1+\alpha/2,2+\alpha}(T,\omega)}
\label{stima-F}
\end{equation}
for every $\bm{w}_1, \bm{w}_2\in B(0,\rho)\subset {\mathscr X}_{1+\alpha/2,2+\alpha}(T,\omega)$.
\item[\rm (ii)]
The function ${\mathcal G}$ maps the ball $B(0,\rho)$ of ${\mathscr X}_{1+\alpha/2,2+\alpha}(T,\omega)$ into $C^{(1+\alpha)/2}([0,T])$. Moreover, there exists a
positive constant $C$ such that
\begin{equation}
\|e^{\omega\cdot}({\mathcal G}(\bm{w}_2)-{\mathcal G}(\bm{w}_1))\|_{C^{(1+\alpha)/2}([0,T])}\le C\rho\|\bm{w}_2-\bm{w}_1\|_{\mathscr{X}_{1+\alpha/2,2+\alpha}(T,\omega)}
\label{stima-G}
\end{equation}
for every $\bm{w}_1, \bm{w}_2\in B(0,\rho)\subset {\mathscr X}_{1+\alpha/2,2+\alpha}(T,\omega)$.
\end{enumerate}
\end{lemma}

\begin{proof}
To prove the estimate in (i) and (ii), we need some preliminary results. More precisely, we need to show that ${\mathscr X}_{1+\alpha/2,2+\alpha}(T,\omega)$ is continuosly embedded both into ${\mathscr X}_{\alpha/2,\alpha}(T,0)$ and into ${\mathscr X}_{(1+\alpha)/2.0}(T,0)$.

To begin with, we show that $\mathscr{X}_{1+\alpha/2.2+\alpha}(T,\omega)$ is continuosly embedded into $\mathscr{X}_{\beta,0}(T,0)$ for every $\beta\in (0,1)$. For this purpose, we recall that, if $g\in C^1([0,T])$, then
\begin{eqnarray*}
|g(t_2)-g(t_1)|\le \|g'\|_{\infty}|t_2-t_1|
\le \|g'\|_{\infty}|t_2-t_1|^{\beta}
\end{eqnarray*}
for every $t_1,t_2\in {\mathcal I}_T$, with $|t_2-t_1|\le 1$, and
\begin{eqnarray*}
|g(t_2)-g(t_1)|\le 2\|g\|_{\infty}
\le \|g\|_{\infty}|t_2-t_1|^{\beta}
\end{eqnarray*}
for every $t_1,t_2\in {\mathcal I}_T$, with $|t_2-t_1|\ge 1$. Hence,
\begin{align}
[\bm{v}]_{{\mathscr X}_{\beta,0}(T,0)}
\le & 2\|e^{\omega\cdot}\bm{v}\|_{\infty}
+\|D_t(e^{-\omega}(e^{\omega\cdot}\bm{v}))\|_{\infty}\notag\\
\le & 2\|e^{\omega\cdot}\bm{v}\|_{\infty}
+\|D_t(e^{\omega\cdot}\bm{v})\|_{\infty}
+\omega\|e^{\omega\cdot}\bm{v}\|_{\infty}\notag\\
\le & (2+\omega)\|\bm{v}\|_{{\mathscr X}_{1+\alpha/2,2+\alpha}(T,\omega)}
\label{v-Holder}
\end{align}
so that
\begin{align}
\|\bm{v}\|_{{\mathscr X}_{\beta,0}(T,0)}\le (3+\omega)\rho.
\label{v-X}
\end{align}

Next, we show that if $\bm{w}\in {\mathscr X}_{1+\alpha/2,2+\alpha}(\omega.T)$, then the spatial $\alpha$-H\"older seminorm of $\bm{w}$ can be estimated in a completely similar (and easier) way. It turns out that 
\begin{align}
[w_{1,0}^{\sharp}(t,\cdot)]_{C^{\alpha}_b(I_-)}
\le &2\|w_{1,0}^{\sharp}\|_{C_b(I_-)}+\|D_y w_{1,\omega}^{\sharp}\|_{C_b(I_-)}\nonumber\\
\le &\frac{1}{2}(5-c_L)\|w_{1,0}^{\sharp}(t,\cdot)\|_{C_b(I_-)}
+\bigg\|\frac{D_yw_1(t,\cdot)}{q_L}\bigg\|_{C_b(I_-)}
\label{A-w}
\end{align}
and
\begin{align}
[w^{\sharp}_{2,0}(t,\cdot)]_{C^{\alpha}_b(I_+)}
\le \frac{1}{2}(3+c_H)\|w^{\sharp}_{2,0}(t,\cdot)\|_{C_b(I_+)}+\bigg\|\frac{D_y\bm{w}(t,\cdot)}{q_H}\bigg \|_{C_b(I_+)}
\label{B-w}
\end{align}
for every $t\in {\mathcal I}_T$.

Combining estimates \eqref{v-X}, \eqref{A-w} and \eqref{B-w}, we deduce that $\bm{w}\in {\mathscr X}_{\alpha/2,\alpha}(T,0)$ and
\begin{align*}
\|\bm{w}\|_{\mathscr{X}_{\alpha/2,\alpha}(T,0)}\le c_*\|\bm{w}\|_{\mathscr{X}_{1+\alpha/2,2+\alpha}(T,\omega)}
\end{align*}
for some positive constant $c_*$, independent of $\bm{w}$. which can be explicitly computed using the quoted estimates.

Finally, we show that, if $\bm{u}\in \mathscr{X}_{1+\alpha/2,2+\alpha}(T,\omega)$, then $D_y\bm{u}\in\mathscr{X}_{\alpha/2,\alpha}(T,0)$.
For this purpose we recall that there exists a positive constant $c_0$ such that
\begin{eqnarray*}
\|g\|_{C^1(I_{\pm
})}
\le c_0\|g\|_{C(I_{\pm
})}^{\frac{1}{2}}\|g\|_{C^2(I_{\pm
})}^{\frac{1}{2}},\qquad\;\,g\in C^2(I_{\pm}),
\end{eqnarray*}
so that we can estimate
\begin{align*}
&\bigg |\frac{D_yw_1(t_2,x)}{q_L(x)}
-\frac{D_yw_1(t_1,x)}{q_L(x)}\bigg |\\
\le & |D_yw_{2,0}^{\sharp}(t_2,x)-D_yw_{2,0}^{\sharp}(t_1,x)|
+\frac{1}{2}(1-c_L)|w_{1,0}^{\sharp}(t_2,x)-w_{1,0}^{\sharp}(t_1,x)|\\
\le &\bigg (\sqrt{2}c_0\|D_tw_{1,0}^{\sharp}\|_{C_b([0,T]\times I_-)}^{\frac{1}{2}}\|w_{1,0}^{\sharp}\|_{C^2_b([0,T]\times I_-)}^{\frac{1}{2}}+\frac{1}{2}(1-c_L)\|D_tw_{1,0}^{\sharp}\|_{C_b([0,T]\times I_-)}\bigg )|t-s|^{\frac{\alpha}{2}}
\end{align*}
for every $x\in I_-$ and every $s,t\in {\mathcal I}_T$ with $|t-s|\le 1$. Moreover,
\begin{align*}
&\bigg |\frac{D_yw_1(t_2,x)}{q_L(x)}
-\frac{D_yw_1(t_1,x)}{q_L(x)}\bigg |
\le 2\bigg\|\frac{D_yw_1}{q_L}\bigg\|_{C([0,T]\times I_-)}|t-s|^{\frac{\alpha}{2}}
\end{align*}
for every $x\in I_-$ and $s,t\in {\mathcal I}_T$, with $|t-s|\ge 1$. A complete similar estimate can be proved when $x\in [\eta^*,+\infty)$, with $1-c_L$ being replaced by $c_H-1$. Further, arguing as in \eqref{A-w} and \eqref{B-w} it can be shown that
\begin{align*}
&\bigg [\frac{D_yw_1(t,\cdot)}{q_L}\bigg ]_{C^{\alpha}_b(I_-)}+\bigg [\frac{D_yw_1(t,\cdot)}{q_H}\bigg ]_{C^{\alpha}_b(I_+)}\\
\le &\frac{1}{2}\max\{5-c_L,3+c_H\}\|D_y\bm{w}(t,\cdot)\|_{\mathscr{X}_{0,0}(T,\omega)}+\|D_{yy}\bm{w}(t,\cdot)\|_{\mathscr{X}_{0,0}(T,\omega)}.
\end{align*}
for every $t\in [0,T]$.
Combining all these estimates, we conclude that
there exists a positive constant $c_{**}$ such that
\begin{align}
\|D_y\bm{w}\|_{{\mathscr X}_{\alpha/2,\alpha}(T,0)}\le c_{**}
\|\bm{w}\|_{{\mathscr X}_{1+\alpha/2,2+\alpha}(T,\omega)},\qquad\;\,\bm{w}\in\mathscr{X}_{1+\alpha/2,2+\alpha}(T,\omega).
\label{Dyv}
\end{align}

(i) Let $\rho_0>0$ be such that
\begin{eqnarray*}
1-\Psi'(w_1(t,\eta^*))\bigg\{\gamma e^{\eta^*}\left (1-\frac{2\delta}{\sigma^2_L}\right )^2\Psi(w_1(t,\eta^*))+D_yw_1(t,\eta^*_-)\bigg\}\ge \frac{1}{2}    
\end{eqnarray*}
for every $\bm{w}\in B(0,\rho)\subset \mathscr{X}_{1+\alpha/2,2+\alpha}(T,\omega)$ and every $\rho\in (0,\rho_0]$. From now on $\rho$ is arbitrarily fixed in the interval $(0,\rho_0]$.

To begin with, we observe that, for every $\bm{v}\in B(0,\rho)\subset {\mathscr X}_{1+\alpha/2,2+\alpha}(T,\omega)$, it holds that
\begin{align*}
e^{\omega t}\FF(\bm{w}(t,\cdot))=&\frac{\frac{\sgL2}{2}\Psi'(w_1(t,\eta^*))D_{yy}e^{\omega t}w_1(t,\eta^*) + \left(\delta-
\frac{1}{2}\sgL2\right)\Psi'(w_1(t,\eta^*))e^{\omega t}D_yw_1(t,\eta^*)}
{1-\Psi'(w_1(t,\eta^*))\left\{\gamma e^{\eta^*}\left (1-\frac{2\delta}{\sigma^2_L}\right )^2\Psi(w_1(t,\eta^*))+D_yw_1(t,\eta^*_-)\right\}} \nonumber\\
&\qquad\;\, \times \bigg(D_y\bm{w}+\gamma e^{\eta^*}\bigg (1-\frac{2\delta}{\sgL2}\bigg )^2\Psi(\bm{w})\bigg)\\
=&\frac{\widetilde{\mathcal{F}}_1(\bm{w}(t,\cdot))}{\widetilde{\mathcal{F}}_2(\bm{w}(t,\cdot))}\widetilde{\mathcal{F}}_3(\bm{w}(t,\cdot))
\end{align*}
for every $t\in {\mathcal I}_T$. 

Using \eqref{v-Holder}, \eqref{v-X} and observing that 
\begin{align*}
\Psi'(w_1(\cdot,\eta^*))-\Psi'(z_1(\cdot,\eta^*))=(w_1(\cdot,\eta^*)-z_1(\cdot,\eta^*))\int_0^1\Psi''(\sigma w_1(\cdot,\eta^*)+(1-\sigma)z_1(\cdot,\eta^*))d\sigma
\end{align*}
we deduce that
\begin{align}
&\|\Psi^{(j)}(w_1(\cdot,\eta^*))-\Psi^{(j)}(z_1(\cdot,\eta^*))\|_{C^{\alpha/2}_b({\mathcal I}_T)}\notag\\
\le &\big [\|\Psi^{(j+1)}\|_{\infty,\rho}+\frac{1}{2}\|\Psi^{(j+2)}\|_{\infty,\rho}q_*([\bm{w}]_{{\mathscr X}_{\alpha/2,0}(T,0)}+
[\bm{z}]_{{\mathscr X}_{\alpha/2,0}(T,0)})\big ]q_*\|\bm{w}-\bm{z}\|_{\mathscr{X}_{\alpha/2,0}(T,0)}\notag\\
\le & \big [\|\Psi^{(j+1)}\|_{\infty,\rho}+(2+\omega)\rho\|\Psi^{(j+2)}\|_{\infty,\rho}q_*\big ](3+\omega)q_*\|\bm{w}-\bm{z}\|_{\mathscr{X}_{1+\alpha/2,2+\alpha}(T,\omega)}\notag\\
=&\!:C_jq_*\|\bm{w}-\bm{z}\|_{\mathscr{X}_{1+\alpha/2,2+\alpha}(T,\omega)}
\label{psi-1}
\end{align}
for every $\bm{w},\bm{z}\in B(0,\rho)\subset {\mathscr X}_{1+\alpha/2,2+\alpha}(T,\omega)$ and every $j=0,1$, where 
$\|\Psi^{(k)}\|_{\infty,\rho}$ denotes the sup-norm of the $k$-th derivative of the function $\psi$ on the interval $[0,\rho q_*]$ and $q_*=\max\{q_L(\eta^*),q_H(\eta^*)\}$.
Analogously,
\begin{align}
\|\Psi^{(j)}(w_1(\cdot,\eta^*))\|_{C^{\alpha/2}({\mathcal I}_T)}
\le & \|\Psi^{(j)}\|_{\infty,\rho}+
\|\Psi^{(j+1)}\|_{\infty,\rho}q_*(2+\omega)\|\bm{w}\|_{{\mathscr X}_{1+\alpha/2,2+\alpha}(T,\omega)}\notag\\
\le & \|\Psi^{(j)}\|_{\infty,\rho}+
\|\Psi^{(j+1)}\|_{\infty,\rho}q_*(2+\omega)\rho=:\widetilde C_j
\label{psi-2}
\end{align}
for every $\bm{w}\in B(0,\rho)\subset \mathscr{X}_{1+\alpha/2,2+\alpha}(T,\omega)$.

Straightforward computations based on \eqref{v-Holder}, \eqref{v-X} and \eqref{Dyv}-\eqref{psi-2}, reveal that
\begin{align*}
\|\widetilde{\mathcal{F}}_1(\bm{w})\|_{C^{\alpha/2}([0,T])}
\le &\frac{1}{2}\widetilde C_1q_*\big [
\sigma_L^2\|D_{yy}\bm{w}\|_{{\mathscr X}_{\alpha/2,0}(T,\omega)}+|2\delta-\sigma_L^2|\|D_y\bm{w}\|_{{\mathscr X}_{\alpha/2,0}(\omega.T)}\big ]\\
\le &
\frac{1}{2}\widetilde C_1q_*\big [
\sigma_L^2+|2\delta-\sigma_L^2|c_*\big ]\rho=:K_1\rho
\end{align*}
\begin{align*}
&\|\widetilde{\mathcal{F}}_1(\bm{w})
-\widetilde{\mathcal{F}}_1(\bm{z})\|_{C^{\alpha/2}([0,T])}\\
\le &\frac{1}{2}C_1q_*^2\big [
\sigma_L^2\|D_{yy}\bm{w}\|_{{\mathscr X}_{\alpha/2,0}(T,\omega)}+|2\delta-\sigma_L^2|\|D_y\bm{w}\|_{{\mathscr X}_{\alpha/2,0}(\omega.T)}\big ]\|\bm{w}-\bm{z}\|_{\mathscr{X}_{1+\alpha/2,2+\alpha}(T,\omega)}\\
&+\frac{1}{2}\widetilde C_1q_*\big [\sigma_L^2\|D_{yy}\bm{w}-D_{yy}\bm{z}\|_{{\mathscr X}_{\alpha/2,0}(T,\omega)}+|2\delta-\sigma_L^2|\|D_y\bm{w}-D_y\bm{z}\|_{{\mathscr X}_{\alpha/2,0}(T,\omega)}\big ]\\
\le &\bigg (\frac{1}{2}C_1q_*^2[\sigma_L^2+|2\delta-\sigma_L^2|c_{**}]\rho+K_1\bigg )\|\bm{w}-\bm{z}\|_{{\mathscr X}_{1+\alpha/2,2+\alpha}(T,\omega)}\\
=&\!: K_2(\rho)\|\bm{w}-\bm{z}\|_{{\mathscr X}_{1+\alpha/2,2+\alpha}(T,\omega)};
\end{align*}
\begin{align*}
\bigg\|\frac{1}{\widetilde{\mathscr{F}}_2(\bm{w})}\bigg\|_{C^{\alpha/2}([0,T])}
\le & 2+4\widetilde C_1\bigg (\gamma e^{\eta^*}\Big (1-\frac{2\delta}{\sigma_L^2}\Big )^2\|\Psi\|_{\infty,\rho}+
q_*\|D_y\bm{w}\|_{\mathscr{X}_{0,0}(T,0)}\bigg )\\
&+4\|\Psi'\|_{\infty,\rho}\bigg (
\gamma e^{\eta^*}\Big (1-\frac{2\delta}{\sigma_L^2}\Big )^2\widetilde C_0+q_*
[D_y\bm{w}]_{\mathscr{X}_{\alpha/2,0}(T,0)}\bigg )\\
\le & 2+4\widetilde C_1
\bigg (2\gamma e^{\eta^*}\Big (1-\frac{2\delta}{\sigma_L^2}\Big )^2\widetilde C_0+
q_*c_{**}\rho\bigg )=:K_3(\rho);
\end{align*}
\begin{align*}
&\bigg\|\frac{1}{\widetilde{\mathscr{F}}_2(\bm{w})}-\frac{1}{\widetilde{\mathscr{F}}_2(\bm{z})}\bigg\|_{C^{\alpha/2}([0,T])}\\
\le & q_*(K_3(\rho))^2
\bigg [
\gamma e^{\eta^*}\Big (1-\frac{2\delta}{\sigma_L^2}\Big )^2(C_1\widetilde C_0+\widetilde C_1C_0)+(C_1q_*+\widetilde C_1\rho)c_{**}
\bigg ]\|\bm{w}-\bm{z}\|_{{\mathscr X}_{1+\alpha/2,2+\alpha}(T,\omega)}\\
=&\!:K_4(\rho)\|\bm{w}-\bm{z}\|_{{\mathscr X}_{1+\alpha/2,2+\alpha}(T,\omega)}.
\end{align*}

To estimate ${\mathcal F}_3(\bm{w})$, we observe that
\begin{align}
\frac{\Psi(w_1(\cdot,x))}{q_L(x)}-\frac{\Psi(z_1(\cdot,x))}{q_L(x)}
=&(w_{1,0}^{\sharp}(\cdot,x)-z_{1,0}^{\sharp}(\cdot,x))\int_0^1\Psi'(\sigma w_1(\cdot,x)+(1-\sigma)z_1(\cdot,x))d\sigma\notag\\
=&\!: (w_{1,0}^{\sharp}(\cdot,x)-z_{1,0}^{\sharp}(\cdot,x)){\mathcal J}(\cdot,x)
\label{10-1}
\end{align}
for every $x\in I_-$.

From this formula and arguing as in the proof of \eqref{psi-1}, it can be easily shown that
\begin{align}
&[{\mathcal J}(\cdot,x)]_{C^{\alpha/2}_b({\mathcal I}_T)}\le \frac{1}{2}\|\Psi''\|_{\infty,\rho}
(w_{1,0}^{\sharp}(\cdot,x)]_{C^{\alpha/2}_b({\mathcal I}_T)}+[
z_{1,0}^{\sharp}(\cdot,x)]_{C^{\alpha/2}_b({\mathcal I}_T)} )q_*
\label{10-2}
\end{align}
for every $x\in (-\infty,\eta^*]$ and
\begin{align*}
&[{\mathcal J}(t,\cdot)]_{C^{\alpha}_b(I_-)}
\le \frac{1}{2}
\|\Psi''\|_{\infty,\rho}
\big ([w_1(t,\cdot)]_{C^{\alpha}(I_-)}+[
z_1(t,\cdot)]_{C^{\alpha}(I_-)}\big ),\qquad\;\,t\in {\mathcal I}_T.
\end{align*}
Observing that
\begin{align*}
[w_1(t,\cdot)]_{C^{\alpha}_b(I_-)}
\le [w_{1,0}^{\sharp}(t,\cdot)]_{C^{\alpha}_b(I_-)}
q_*+\|w_{1,0}^{\sharp}(t,\cdot)\|_{\infty}[q_L]_{C^{\alpha}_b(I_-)}
\end{align*}
and
\begin{align}
|q_L(x_2)-q_L(x_1)|\le & \bigg |\frac{1}{2}(1-c_L)\int_{x_1}^{x_2}e^{-\frac{1}{2}(c_L-1)r}dr\bigg |\notag\\
\le &\frac{1}{2}(1-c_L)|x_2-x_1|^{\alpha}\bigg (\int_{x_1}^{x_2}e^{-\frac{(c_L-1) r}{2(1-\alpha)}}dr\bigg )^{1-\alpha}\notag\\
\le &\bigg (\frac{1-c_L}{2}\bigg )^{\alpha}(1-\alpha)^{1-\alpha}|x_2-x_1|^{\alpha}
\le \bigg (\frac{1-c_L}{2}\bigg )^{\alpha}|x_2-x_1|^{\alpha}
\label{servo}
\end{align}
for every $x_1,x_2\in (-\infty,\eta^*]$. Hence,
\begin{align}
[{\mathcal J}(t,\cdot)]_{C^{\alpha}_b(I_-)}\le \frac{1}{2}\|\Psi''\|_{\infty,\rho}\bigg [
&q_*\big ([w_{1,0}^{\sharp}(t,\cdot)]_{C^{\alpha}_b(I_-)}+ [z_{1,0}^{\sharp}(t,\cdot)]_{C^{\alpha}_b(I_-)}\big )\notag\\
&+\bigg (\frac{1-c_L}{2}\bigg )^{\alpha}\big (
\|w_{1,0}^{\sharp}(t,\cdot)\|_{\infty}+\|z_{1,0}^{\sharp}(t,\cdot)\|_{\infty}\big )\bigg  ]
\label{10-3}
\end{align}

From \eqref{10-1}-\eqref{10-3}, we conclude that
\begin{align*}
&\bigg\|\frac{\Psi(w_1)}{q_L}-\frac{\Psi(z_1)}{q_L}\bigg\|_{C^{\alpha/2,\alpha}_b({\mathcal I}_T\times I_-)}\\
\le &\|w_{1,0}^{\sharp}-
z_{1,0}^{\sharp}\|_{C^{\alpha/2,\alpha}_b({\mathcal I}_T\times I_-)}\|\Psi'\|_{\infty,\rho}\\
&+\frac{1}{2}\|\Psi''\|_{\infty,\rho}\|w_{1,0}^{\sharp}-z_{1,0}^{\sharp}\|_{\infty}
\bigg (\sup_{x\le\eta^*}[w_{1,0}^{\sharp}(\cdot,x)]_{C^{\alpha/2}_b({\mathcal I}_T)}+\sup_{x\le\eta^*}[z_{1,0}^{\sharp}(\cdot,x)]_{C^{\alpha/2}_b({\mathcal I}_T)}\\
&\qquad\qquad\qquad\qquad\qquad\qquad\quad+
\sup_{t\in {\mathcal I}_T}[w_{1,0}^{\sharp}(t,\cdot)]_{C^{\alpha}_b(I_-)}+\sup_{t\in {\mathcal I}_T}[z_{1,0}^{\sharp}(t,\cdot)]_{C^{\alpha}_b(I_-)}\bigg ).
\end{align*}
In a completely similar way, we can estimate the function
$\frac{\Psi(w_1)}{q_H}-\frac{\Psi(z_1)}{q_H}$.

Putting everything together and observing that
\begin{align*}
[e^{-\omega\cdot}w_{1,0}^{\sharp}(\cdot,x)]_{C^{\alpha/2}_b({\mathcal I}_T)}\le & [w_{1,0}^{\sharp}(\cdot,x)]_{C^{\alpha/2}_b({\mathcal I}_T)}
+\|w_{1,0}^{\sharp}\|_{\infty}[e^{-\omega\cdot}]_{C^{\alpha/2}_b({\mathcal I}_T)}\\
\le &[w_{1,0}^{\sharp}(\cdot,x)]_{C^{\alpha/2}_b({\mathcal I}_T)}
+\|w_{1,0}^{\sharp}\|_{\infty}\omega^{1-\alpha},
\end{align*}
where we have adapted \eqref{servo} to estimate $[e^{-\omega\cdot}]_{C^{\alpha/2}({\mathcal I}_T)}$, and a similar estimate holds true with $w_1$ being replaced by $z_1$, we conclude that
\begin{align*}
&\|\widetilde{\mathscr{F}}_3(\bm{w})-
\widetilde{\mathscr{F}}_3(\bm{z})\|_{\mathscr{X}_{\alpha/2,\alpha}(T,0)}\\
\le & \bigg\{c_*+\gamma e^{\eta^*}\Big (1-\frac{2\delta}{\sigma_L^2}\Big )^2
\bigg [\|\Psi'\|_{\rho.\infty}c_*
+(2+\omega)\rho q_*\|\Psi''\|_{\infty,\rho}
+2\rho c_{H,L}^{\alpha}\bigg ]\bigg\}\|\bm{w}-\bm{z}\|_{{\mathscr X}_{1+\alpha/2,2+\alpha}(T,\omega)}\\
=&\!: K_5(\rho)\|\bm{w}-\bm{z}\|_{{\mathscr X}_{1+\alpha/2,2+\alpha}(T,\omega)},
\end{align*}
where $c_{H,L}=\max\left\{\frac{1-c_L}{2},\frac{c_H-1}{2}\right\}$.

Summing up, we have proved that
\begin{eqnarray*}
\|{\mathcal F}(\bm{w})-\mathcal{F}(\bm{z})\|_{{\mathscr X}_{\alpha/2,\alpha}(T,\omega)}\le K_6(\rho)\rho
\|\bm{w}-\bm{z}\|_{{\mathscr X}_{1+\alpha/2,2+\alpha}(T,\omega)},
\end{eqnarray*}
where $K_6(\rho)=\big [K_2(\rho)K_3(\rho)
+\rho K_1(K_3(\rho)+K_4(\rho))\big ]K_5(\rho)$
and \eqref{stima-F} follows at once.

(ii) We observe observe that, since $R(0)=R'(0)=0$, we can write
\begin{align}
e^{\omega\cdot}[R(w_1(\cdot,\eta^*))-R(z_1(\cdot,\eta^*))]=&e^{\omega\cdot}(w_1(\cdot,\eta^*)-z_1(\cdot,\eta^*))\int_0^1R'(\sigma w_1(\cdot,\eta^*)+(1-\sigma)z_1(\cdot,\eta^*))d\sigma\notag\\
=&e^{\omega\cdot}(w_1(\cdot,\eta^*)-z_1(\cdot,\eta^*))\int_0^1(\sigma w_1(\cdot,\eta^*)+(1-\sigma)z_1(\cdot,\eta^*))d\sigma\notag\\
&\qquad\qquad\qquad\qquad\times\int_0^1R''(\tau\sigma w_1(\cdot,\eta^*)+\tau(1-\sigma)z_1(\cdot,\eta^*))d\tau.
\label{stima-G1}
\end{align}
Using the second equality in \eqref{stima-G1} to estimate the sup-norm of the function
$e^{\omega\cdot}[R(w_1(\cdot,\eta^*))-R(z_1(\cdot,\eta^*))]$ and the first equality to estimate its $(1+\alpha)$-H\"older seminorm, we easily conclude that
\begin{align*}
&\|e^{\omega\cdot}[R(w_1(\cdot,\eta^*))-R(z_1(\cdot,\eta^*))]\|_{C^{(1+\alpha)/2}_b({\mathcal I}_T)}\\
\le &\frac{1}{2}q_*^2\|\bm{w}-\bm{z}\|_{{\mathscr X}_{(1+\alpha)/2,0}(T,0)}\|R''\|_{\infty,\rho}(\|\bm{w}\|_{{\mathscr X}_{(1+\alpha)/2,0}(T,0)}+\|\bm{z}\|_{{\mathscr X}_{(1+\alpha)/2,0}(T,0)}).
\end{align*}

Taking \eqref{v-X} into account, we from the previous estimate we deduce that
\begin{align*}
&\|{\mathcal G}(\bm{w})-{\mathcal G}(\bm{z})\|_{C^{(1+\alpha)/2}_b({\mathcal I}_T)}\\
\le &q_*^2c_{**}\gamma e^{\eta^*}\bigg (1-\frac{2\delta}{\sigma_L^2}\bigg )\frac{\delta}{\sigma_H^2\sigma_L^2}(\sigma_L^2-\sigma_H^2)\|R''\|_{\infty,\rho}(2+\omega)\rho\|\bm{w}-\bm{z}\|_{{\mathscr X}_{(1+\alpha)/2,0}(T,\omega)}
\end{align*}
and \eqref{stima-G} follows.

The proof is now complete.
\end{proof}

Based on these remarks, we are now in a position to prove the following theorem.

\begin{theorem}\label{stability}
Fix $\alpha\in [0,1)$. Then, the null solution to problem \eqref{FNLP} is stable in the ${\mathscr X}_{2+\alpha}$-norm. More precisely, for every $\omega_0\in \left (0,\min\left\{\frac{c_L^2}{8}(c_L^2-1)^2,\frac{c_H^2}{8}(c_H^2-1)^2\right\}\right )$ there exists a positive constant $\rho$ such that, for every initial datum $\bm{w}_0\in B(0,\rho)\subset {\mathscr X}_{2+\alpha}$, satisfying the compatibility conditions
\begin{eqnarray*}
{\mathcal B}(\bm{w}_0)={\mathcal G}(\bm{w}_0),\qquad\;\,B_0({\mathcal L}\bm{w}_0+{\mathcal F}(\bm{w_0}))=0,
\end{eqnarray*}
the solution $\bm{w}$ to problem \eqref{FNLP}, with initial datum $\bm{w}_0$, exists for all the positive times, belongs to ${\mathscr X}_{1+\alpha/2,2+\alpha}(+\infty,\omega_0)$ and there exists a positive constant $C$ such that
\begin{equation}
\|\bm{w}(t,\cdot)\|_{\mathscr{X}_{2+\alpha}}\le Ce^{-\omega_0t}\|\bm{w}_0\|_{\mathscr{X}_{2+\alpha}},\qquad\;\,t>0.
\label{stima-finale}
\end{equation}
\end{theorem}

\begin{proof}
Since it is rather long, we split the proof into several steps. We provide the proof in the case when $\alpha\in (0,1)$, since in the case $\alpha=0$ is similar and even simpler. Hence, 
throughout the proof, $\alpha$ is arbitrarily fixed in $(0,1)$ and we use the notation introduced at the beginning of this section.

{\em Step 1}. Here, we prove that for every $T>0$, the Cauchy problem
\begin{equation}\label{FNLP-L}
\left\{
\begin{array}{ll}
\displaystyle \frac{d\bm{w}}{d t}(t,\cdot)= \LL\bm{w}(t,\cdot) + \bm{f}(t,\cdot), & t\in [0,T],\\[3mm]
({\mathcal B}\bm{w})(t)= (0,g(t)), & t\in [0,T],\\[1mm]
\bm{w}(0,\cdot)=\bm{w}_0,
\end{array}
\right.
\end{equation}
admits a unique solution $\bm{w}\in\mathscr{X}_{1+\alpha/2,2+\alpha}(T,0)$, for every $\bm{w}_0\in\mathscr{X}_{2_\alpha}$, every $\bm{f}=(f_1,f_2)\in {\mathscr X}_{\alpha/2,\alpha}(T,0)$ and for every function $g\in C^{\frac{1+\alpha}{2}}([0,T];\R)$, satisfying the compatibility conditions
\begin{eqnarray*}
{\mathcal B}\bm{w}_0=\bm{g}(0),\qquad\;\,B_0(\LL\bm{w}_0+f(0
))=0.
\end{eqnarray*}
Moreover, there exists a positive constant $C$ such that
\begin{equation}
\|{\bm w}\|_{\mathscr{X}_{1+\alpha/2,2+\alpha}(T,0)}\le C\big (\|\bm{w}_0\|_{\mathscr{X}_{2+\alpha}}+\|\bm{f}\|_{{\mathscr X}_{\alpha/2,\alpha}(T,0)}+\|g\|_{C^{1+\alpha}([0,T])}\big ),
\label{stima-xxx}
\end{equation}
where the constant $C$ depends on only on 
the quantity $M_k=\sup_{t\in [0,T+1]}\|t^kL^ke^{tL}\|_{L(\mathscr{X})}$ ($k=0,1,2$),
$M_{k,\beta}=\sup_{t\in [0,T+1]}\|t^{k-\beta}L^ke^{tL}\|_{L(D_L(\beta,\infty),\mathscr{X})}$ $(k=0,1,2,3)$  suitable choices of $\beta$.

For this purpose, we begin by introducing a lifting operator ${\mathcal N}$, defined by ${\mathcal N}a=(a\zeta,0)$ for every $a\in\R$, where $\zeta$ is a smooth function, which vanishes at
$x=\eta^*$ and satisfies the condition $\zeta'(\eta^*)=1$.
A straightforward computation reveals that ${\mathcal B}{\mathcal N}a=(0,a)$ for every $a\in\R$.

The candidate to be the solution to the above Cauchy problem is the function $\bm{w}$, defined by
\begin{align}
\bm{w}(t,\cdot)=e^{tL}\bm{w}_0+\int_0^te^{(t-s)L}\big (f(s,\cdot)+\LL{\mathcal N}g(s,\cdot))ds-L\int_0^te^{(t-s)L}({\mathcal N}g(s,\cdot))ds,
\label{pitbike}
\end{align}
for every $t\in [0,T]$.

To check all the above properties, we begin by introducing the function $\bm{z}:[0,T]\to\mathscr{X}$, defined by
\begin{eqnarray*}
\bm{z}(t)=\int_0^te^{(t-s)L}[{\mathcal N}(g(s))-{\mathcal N}(g(0))]ds,\qquad\;\,t\in [0,T].
\end{eqnarray*}

Since the function $t\mapsto {\mathcal N}(g(s))-{\mathcal N}(g(0))$ belongs to $C^{(1+\alpha)/2}([0,T];\mathscr{X}_2)$ which is continuously embedded into $C^{(1+\alpha)/2}([0,T];D_L(1/2,\infty))$, from
\cite[Theorem 4.3.16]{Lunardi96}, it follows that the function $\bm{z}$ belongs to $C^1([0,T];\mathscr{X})\cap C^2([0,T];D(L))$ and solves the equation $\bm{z}'(t)=L\bm{z}(t)+{\mathcal N}(g(t))-{\mathcal N}(g(0))$
for every $t\in [0,T]$. Moreover, $L\bm{z}$ belongs to $C^{1+\frac{\alpha}{2}}([0,T];\mathscr{X})$, $\bm{z}'$ is bounded in $[0,T]$ with values in $D_L(1+\alpha/2,\infty)$ and there exists a positive constant $C$, independent of $\bm{w}$ such that
\begin{equation}
\|\bm{z}\|_{C^1([0,T];\mathscr{X})}+\|L\bm{z}\|_{C([0.T];\mathscr{X})}+\|\bm{z}'\|_{B([0,T];D_L(1+\alpha/2,\infty))}\le C\|g\|_{C^{\frac{1+\alpha}{2}}([0,T])}
\label{stima-aaa}
\end{equation}

By interpolation, it follows that
$\bm{z}'$ is continuous with values in $D(L)$. Writing,
\begin{eqnarray*}
\bm{z}(t)-\bm{z}(t_0)=\int_{t_0}^t\bm{z}'(s)ds,\qquad\;\,t,t_0\in [0,T],
\end{eqnarray*}
and applying $L$ to both sides of the previous equation, we conclude that
\begin{eqnarray*}
L\bm{z}(t)-L\bm{z}(t_0)=\int_{t_0}^tL\bm{z}'(s)ds,\qquad\;\,t,t_0\in [0,T].
\end{eqnarray*}
Hence, function $L\bm{z}$ is differentiable in $[0,T]$ with values in $\mathscr{X}$ and $(L\bm{z})'=L\bm{z}'=\mathcal{L}L\bm{z}+{\mathcal L}{\mathcal N}(g)-{\mathcal L}{\mathcal N}(g(0))$.

Let us now introduce the function ${\bm w}_1=-L\bm{z}+\mathscr{N}(g(0))$. Since $(L\bm{w}_1)'\in C^{\frac{\alpha}{2}}([0,T];\mathscr{X})\cap B([0,T];D_L(\alpha/2,\infty))$, the function
$(L\bm{w}_1)'$ belongs to $\mathscr{X}_{\alpha/2,\alpha}(T,0)$. As a byproduct, $\bm{w}_1$ belongs to this space. Further, we observe that $L\bm{z}=\bm{z}'-{\mathcal N}(g)+{\mathcal N}(g(0))$, so that
$L\bm{z}$ is bounded in $[0,T]$ with values in ${\mathscr X}_{2+\alpha}$, due to the characterization of the interpolation space $D_L(1+\alpha/2,\infty)$ in Theorem \ref{thm-7.7}.
Moreover,
\begin{align*}
\bm{w}_1'(t)=-(L{\bm z})'(t)=-\mathcal{L}L\bm{z}(t)-{\mathcal L}{\mathcal N}(g(t))+{\mathcal L}{\mathcal N}(g(0))={\mathcal L}\bm{w}_1(t)-{\mathcal L}{\mathcal N}(g(t))
\end{align*}
for every $t\in [0,T]$.

To get the smoothness of the spatial derivatives of $\bm{w}_1=(w_1,w_2)$ with respect to the time variables, we write
\begin{eqnarray*}
w_1(t,x)-w_1(t_0,x)=\int_{t_0}^tD_tw_1(s,x)ds,\qquad\;\, t_0,t\in [0,T]\;\,x\in (-\infty,\eta^*],
\end{eqnarray*}
which allows us to easily prove that
\begin{align*}
\|w_{1,0}^{\sharp}(t,\cdot)-w_{1,0}^{\sharp}(t_0,\cdot)\|_{C^{\alpha}_b(I_-)}
\le &C|t-t_0|\sup_{s\in [0,T]}\|w_{1,0}^{\sharp}(t,\cdot)\|_{C^{\alpha}_b(I_-)}\\
\le &C|t-t_0|\|D_t{\bm w}_1\|_{{\mathscr X}_{\alpha/2,\alpha}(T,0)}
\end{align*}
for some positive constant $C$, independent of $\bm{w}_1$. In the same way we can show that
\begin{eqnarray*}
\|w_{2,0}^{\sharp}(t,\cdot)-w_{2,0}^{\sharp}(t_0,\cdot)\|_{C^{\alpha}_b(I_+)}
\le C|t-t_0|\|D_t{\bm w}_2\|_{{\mathscr X}_{\alpha/2,\alpha}(T,0)}.
\end{eqnarray*}

Since $\mathcal{X}_2$ belongs to the class $J_{1-\alpha/2}$ between ${\mathscr X}_{\alpha}$ and ${\mathscr X}_{2+\alpha}$, it follows that there exist positive constants $\widetilde C$ and $C_1$, independent of $\bm{w}_1$, such that
\begin{align*}
\|\bm{w}_1(t,\cdot)-\bm{w}_1(t_0)\|_{{\mathscr X}_2}\le &\widetilde C\|\bm{w}_1(t,\cdot)-\bm{w}_1(t_0)\|_{{\mathscr X}_{\alpha}}^{\frac{\alpha}{2}}\|\bm{w}_1(t,\cdot)-\bm{w}_1(t_0)\|_{{\mathscr X}_{2+\alpha}}^{1=\frac{\alpha}{2}}\\
\le & C_1\sup_{t\in [0,T]}\|\bm{w}_1(t,\cdot)\|_{{\mathscr X}_{2+\alpha}}^{1-\frac{\alpha}{2}}\|D_t\bm{w}_1\|_{\mathscr{X}_{\alpha/2,\alpha}(T,0)}|t-t_0|^{\frac{\alpha}{2}}
\end{align*}
for every $t,t_0\in [0,T]$. From estimate \eqref{stima-aaa} and the previous arguments it follows that
\begin{equation}
\|\bm{w}_1\|_{{\mathscr X}_{1+\alpha/2,2+\alpha}(T,0)}\le C\|g\|_{C^{\frac{1+\alpha}{2}}([0,T])}
\label{stimaaa-1}
\end{equation}
for some positive constant $C$, independent of $\bm{w}_1$.

Summing up, we have proved that $\bm{w}_1$ belongs to ${\mathscr X}_{1+\alpha/2,2+\alpha}(T,0)$. Moreover, since $\bm{z}'(t)\in D(L)$ for every $t\in [0,T]$ and $\bm{z}'=L\bm{z}+{\mathcal N}(g(t))-{\mathcal N}(g(0))=-\bm{w}_1+{\mathcal N}(g(t))$, it follows that
$0=\mathscr{B}(-\bm{w}_1(t)+{\mathcal N}(g(t))=-{\mathscr B}\bm{w}_1(t)+{\mathcal N}(g(t))=-{\mathscr B}\bm{w}_1(t)+g(t)$ for every $t\in [0,T]$. Therefore, $\bm{w}_1$ is a strict solution to the Cauchy problem
\begin{equation*}
\left\{
\begin{array}{ll}
\displaystyle \frac{d\bm{w}_1}{d t}(t,\cdot)= \LL\bm{w}_1(t,\cdot)-{\mathcal L}{\mathcal N}(g(t)), & t\in [0,T],\\[3mm]
({\mathcal B}\bm{w}_1)(t)= (0,g(t)), & t\in [0,T],\\[1mm]
\bm{w}_1(0,\cdot)=\bm{0}.
\end{array}
\right.
\end{equation*}

Finally, we consider the Cauchy problem
\begin{equation}
\left\{
\begin{array}{ll}
\displaystyle \frac{d\bm{w}_2}{d t}(t,\cdot)= \LL\bm{w}_2(t,\cdot)+\bm{f}+{\mathcal L}{\mathcal N}(g(t)), & t\in [0,T],\\[3mm]
({\mathcal B}\bm{w}_2)(t)= \bm{0}, & t\in [0,T],\\[1mm]
\bm{w}_2(0,\cdot)=\bm{w}_0-{\mathcal N}(g(0)),
\end{array}
\right.
\label{FNLP-L2}
\end{equation}

Since $\bm{u}_0\in {\mathscr X}_2$ and ${\mathcal B}(\bm{w}_0-{\mathcal N}(g(0)))={\mathcal B}\bm{w}_0-(0,g(0))=\bm{0}$, due to the compatibility conditions, it follows that $\bm{w}_0-{\mathcal N}(g(0))$ belongs to $D(L)$.
Moreover, $\bm{f}+{\mathcal L}{\mathcal N}(g(t))\in {\mathscr X}_{\alpha/2,\alpha}(T,0)=B([0,T];{\mathscr X}_{\alpha})\cap C^{\alpha/2}([0,T];\mathscr{X})$ and
$B_0({\mathcal L}(\bm{w}_0-{\mathcal N}(g(0)))+\bm{f}(0)+{\mathcal L}{\mathcal N}(g(0)))=B_0({\mathcal L}\bm{w}_0+\bm{f}(0))=0$, again by the compatibility conditions. By classical results for abstract Cauchy problems associated with sectorial operators, we conclude that problem \eqref{FNLP-L2} admits a unique classical solution $\bm{w}_2$ which actually belongs to $\mathscr{X}_{1+\alpha/2,2+\alpha}(T,0)$ and satisfies the estimate
\begin{eqnarray}
\|\bm{w}_2\|_{{\mathscr X}_{1+\alpha/2,2+\alpha}(T,0)}\le C_2(\|{\bm w}_0\|_{{\mathscr X}_{2+\alpha}}+\|\bm{f}\|_{\mathscr{X}_{\alpha/2,\alpha}(T,0)}).
\label{stima-bbb}
\end{eqnarray}

The function $\bm{w}=\bm{w}_1+\bm{w}_2$ is easily seen to be the (unique) strict solution to the Cauchy problem \eqref{FNLP-L} and, in view of estimates \eqref{stimaaa-1} and \eqref{stima-bbb},
\eqref{stima-xxx} follows at once. To complete this step, we observe that
\begin{align*}
\bm{w}_1(t)=&-L\int_0^te^{(t-s)L}({\mathcal N}(g(s))-{\mathcal N}(g(0)))ds+{\mathcal N}(g(0))\\
=&-L\int_0^te^{(t-s)L}{\mathcal N}(g(s))ds+L\int_0^te^{(t-s)L}{\mathcal N}(g(0))ds+{\mathcal N}(g(0))\\
=&-L\int_0^te^{(t-s)L}{\mathcal N}(g(s))ds+e^{tL}{\mathcal N}(g(0))
\end{align*}
for every $t\in [0,T]$. Moreover,
\begin{eqnarray*}
\bm{w}_2(t)=e^{tL}(\bm{u}_0-{\mathcal N}(g(0)))+\int_0^te^{(t-s)L}[\bm{f}(s)+{\mathcal L}{\mathcal N}(g(s))]ds
\end{eqnarray*}
for every $t\in [0,T]$. so that the function $\bm{w}=\bm{w}_1+\bm{w}_2$ is given by \eqref{pitbike}.

{\em Step 2.} In this step, we consider the nonlinear problem
\begin{equation}\label{FNLP-NL}
\left\{
\begin{array}{ll}
\displaystyle \frac{d\bm{w}}{d t}(t,\cdot)= \LL\bm{w}(t,\cdot) + {\mathcal F}(\bm{w}(t,\cdot)), & t\in [0,T],\\[3mm]
({\mathcal B}\bm{w})(t)= (0,{\mathcal G}(\bm{w}(t,\cdot))), & t\in [0,T],\\[3mm]
\bm{w}(0,\cdot)=\bm{w}_0,
\end{array}
\right.
\end{equation}
where ${\mathcal F}(\bm{w})$ and ${\mathcal G}(\bm{w})$ are defined by 
\eqref{Fscript-1} and \eqref{Gscript-1}.

By the results in Step 1, the solution to the previous problem is a fixed point of the operator
$\Gamma: {\mathscr Y}\to {\mathscr Y}$ defined by
\begin{align*}
(\Gamma(\bm{w}))(t,\cdot)=&e^{tL}\bm{w}_0+\int_0^te^{(t-s)L}\big ({\mathcal F}(\bm{w}(s,\cdot))+\LL{\mathcal N}{\mathcal G}(\bm{w}(s,\cdot)))ds-L\int_0^te^{(t-s)L}{\mathcal N}{\mathcal G}(\bm{w}(s,\cdot))ds
\end{align*}
for every $t\in [0,T]$, where ${\mathscr Y}=\overline{B(0,\rho)}\subset\mathscr{X}_{1+\alpha/2,2+\alpha}(\omega_0,\infty)$ endowed with the norm
of ${\mathscr X}_{1+\alpha/2,2+\alpha}(\omega_0,\infty)$, where $\rho>0$ needs to be properly fixed.

Note that
\begin{align*}
e^{\omega_0t}(\Gamma(\bm{w}))(t,\cdot)=&e^{t(L+\omega_0)}\bm{w}_0+\int_0^te^{(t-s)(L+\omega_0)}\big (e^{\omega_0 s}{\mathcal F}(\bm{w}(s,\cdot))+\LL{\mathcal N}(e^{\omega_0s}{\mathcal G}(\bm{w}(s,\cdot))))ds\\
&-L\int_0^te^{(t-s)(L+\omega_0)}{\mathcal N}(e^{\omega_0s}{\mathcal G}(\bm{w}(s,\cdot)))ds
\end{align*}

Using estimates \eqref{stima-F} and \eqref{stima-G} and \eqref{stima-xxx}, where the constant $C$ therein appearing can be made independent of $T$ since it depends only on 
$M_{k,T}$ $(k=0,1,2)$ and $M_{K,\beta,+\infty}$ $(k=0,1,2,3)$, which can be estimated from above by $=\sup_{t\ge 0}\|t^k(L+\omega_0)^ke^{t(L+\omega_0)}\|_{L(\mathscr{X})}$
and $\sup_{t\ge 0}\|t^{k-\beta}(L+\omega_0)^ke^{t(L-\omega)}\|_{L(D_L(\beta,\infty),\mathscr{X})}$, respectively (these constants are finite since the semigroup $(e^{t(L+\omega)})$ is of negative type)
it is immediate to check that $\Gamma$ is a contraction from ${\mathscr Y}$ into itself, provided that $\|\bm{w}_0\|_{{\mathscr X}_{2+\alpha}}\le\rho_0$ and $\rho_0$, $\rho$ are sufficiently small.
Hence, for every $\bm{w_0}\in\overline{B(0,\rho_0)}\subset\mathscr{X}_{2+\alpha}$, problem \eqref{FNLP-NL} admits a unique solution which belongs to 
$\mathscr{X}_{1+\alpha/2,2+\alpha}(\omega_0,+\infty)$, with $\mathscr{X}_{1+\alpha/2,2+\alpha}(\omega_0,+\infty)$-norm which does not exceed 
$\rho$. Using a standard unique continuation argument, it is easy to infer that $\bm{u}$ is is the unique solution which belongs to ${\mathscr X}_{1+\alpha/2,2+\alpha}(\omega_0,+\infty)$. Estimate \eqref{stima-finale} follows at once.
\end{proof}

\section{Convergence to the attenuated traveling wave}\label{stability_phi}

In this section, collecting all the previous results we go back to the original problem \eqref{equ_phi} and \eqref{interface_phi} and state the main result of the paper. To enlighten the notation given a function $\psi$, we write $\psi^{\sharp}$ both to denote the functions
$\frac{\psi}{q_L}$ and $\frac{\psi}{q_H}$. No confusion may arise since 
when we use the notation $\psi^{\sharp}$ on a left-halfline it means that we are considering the function $\frac{\psi}{q_L}$ and when the use the same notation on a right-halfline it means that we are considering the function $\frac{\psi}{q_H}$.

\begin{theorem}
For every $\omega_0\in \left (0,\min\left\{\frac{c_L^2}{8}(c_L^2-1)^2,\frac{c_H^2}{8}(c_H^2-1)^2\right\}\right )$ there exists $\rho_0>0$ such that, for each $\phi_0$ such that
the function $\phi_0^{\sharp}-K$ belongs to $C^{2+\alpha}_b((-\infty,\eta(0)])$ and to $C^{2+\alpha}_b([\eta(0),+\infty))$
and $\|\psi_0^{\sharp}-K\|_{C^{2+\alpha}_b((-\infty,\eta(0)])}+
\|\psi_0^{\sharp}-K\|_{C^{2+\alpha}_b([\eta(0),+\infty))}\le\rho_0$,
the Cauchy problem \eqref{equ_phi}, \eqref{interface_phi} admits a unique solution $(\phi,\eta)$, such that in the variable $t$ and $y=x+ct-\eta(t)+\eta^*$, $u$ belongs to ${\mathscr X}_{1+\alpha/2,2+\alpha}(\omega_0,+\infty)$. Moreover, there exists a positive constant $C$ such that
\begin{equation}
|s(t)+ct-\eta^*|\le Ce^{-\omega_0t}|\phi_0(\eta(0))-K(\eta(0))|,\qquad\;\,t\ge 0,
\label{main-1}
\end{equation}
and
\begin{align*}
\|\phi^{\sharp}(t,\cdot)-e^{-rt}K^{\sharp}\|_{C^{2+\alpha}_b((-\infty,s(t)])}\le &
Ce^{-\sigma_-t}
\big (\|\phi_0^{\sharp}-K^{\sharp}\|_{C^{2+\alpha}_b((-\infty,s(0)])}
+\|\phi_0^{\sharp}-K^{\sharp}\|_{C^{2+\alpha}_b([s(0),+\infty)}\bigg )
\end{align*}
and
\begin{align*}
\|\phi^{\sharp}(t,\cdot)-e^{-rt}K^{\sharp}\|_{C^{2+\alpha}_b([s(t),+\infty))}\le Ce^{-\sigma_+t}
\big (\|\phi_0^{\sharp}-K^{\sharp}\|_{C^{2+\alpha}_b((-\infty,s(0)])}
+\|\phi_0^{\sharp}-K^{\sharp}\|_{C^{2+\alpha}_b([s(0),+\infty)}\bigg )
\end{align*}
for every $t\ge 0$, where $\sigma_-=-\frac{1-C_L}{2}c+\omega_0+r$ and $\sigma_+=\frac{1-c_H}{2}c+\omega_0+r$ are positive numbers.
\end{theorem}

\begin{proof}
Throughout the proof, by $C$ we denote a positive constant, which can vary from line to line but it is independent of $t\ge 0$ and the functions that will appear.

To begin with, we observe that, by \eqref{ansatz} and\eqref{invert}, the smoothness of $\Psi$ and Theorem \ref{stability}, it follows that the function $\bm{v}=(v_1,v_2)$, where
\begin{equation}
v_j(t,\cdot)=\Psi(w(t,\eta^*))K'+ w_j(t,\cdot),\qquad\;\,t>0,\;\,j=1,2,
\label{piani-studio}
\end{equation}
belongs to $\mathscr{X}_{1+\alpha/2,2+\alpha}(\omega,+\infty)$ and
\begin{align}
\|\bm{v}(t,\cdot)\|_{\mathscr{X}_{2+\alpha}}
\le |\Psi(w_1(t,\eta^*))|\bigg (\bigg\|\frac{K'}{q_H}\bigg\|_{C^{2+\alpha}_b((-\infty,\eta^*])}
+\bigg\|\frac{K'}{q_L}\bigg\|_{C^{2+\alpha}_b([\eta^*,+\infty))}\bigg )
+Ce^{-\omega_0t}\|\bm{w}_0\|_{\mathscr{X}_{2+\alpha}}
\label{stima-v}
\end{align}
for every $t\ge 0$. Now, since
\begin{align*}
\Psi(w_1(t,\eta^*))=\Psi(w_1(t,\eta^*))-\Psi(0)
=w_1(t,\eta^*)\int_0^1\Psi'(\sigma w_1(t,\eta^*))d\sigma,\qquad\;\,t\ge 0
\end{align*}
we deduce that
$|\Psi(w_1(t,\eta^*))|\le Ce^{-\omega_0t}\|\bm{w}_0\|_{\mathscr X}$ for every $t\ge 0$.
Hence,
\begin{align*}
\|\bm{v}(t,\cdot)\|_{\mathscr{X}_{2+\alpha}}
\le Ce^{-\omega_0 t}\|\bm{w}_0\|_{\mathscr{X}_{2+\alpha}},\qquad\;\,t\ge 0.
\end{align*}
To get rid of the norm of $\bm{w}_0$ from the right-hand side of the previous estimate, we denote by $\bm{w}_0=(w_{0,1},w_{0,2})$ the value of $\bm{w}$ at $t=0$ and, similarly, by $\bm{v}_0=(v_{0,1},v_{0,2})$ we denote the value of $\bm{v}$ at $t=0$. From formula \eqref{piani-studio} we can infer that
$\bm{v}_0=\bm{w}_0+\Psi(w_{0,1}(\eta^*))K'$, 
so that
\begin{align*}
v_{0,1}(\eta^*)=w_{0,1}(\eta^*)+\Psi(w_{0,1}(\eta^*))K'(\eta^*)
=w_{0,1}(\eta^*)+\gamma e^{\eta^*}\left (1-\frac{2\delta}{\sigma_L^2}\right )\Psi(w_{0,1}(\eta^*)).
\end{align*}

Since $1-\frac{2\delta}{\sigma^2_L}=1-c_L>0$ by \eqref{admissible_bis}, the function $x\mapsto \Lambda(x):=x+\gamma e^{\eta^*}\left (1-\frac{2\delta}{\sigma^2_L}\right )\Psi(x)$ is strictly increasing in a neighborhood of the origin. Hence, $\Lambda$ is invertible in a neighborhood of the origin and we can write $w_{0,1}(\eta^*)=\Lambda^{-1}(v_{0,1}(\eta^*))$. 
We have so proved that 
\begin{eqnarray*}
\bm{w}_0=\bm{v}_0-\Psi(\Lambda^{-1}(v_{0,1}(\eta^*)))K'
\end{eqnarray*}
The functions $\Lambda^{-1}$ and $\Psi$ vanish at zero and are smooth in a neighborhood of the origin. From this remark, we obtain that
$\|\bm{w}_0\|_{\mathscr{X}_{2+\alpha}}\le C\|\bm{v}_0\|_{\mathscr{X}_{2+\alpha}}$, so that, from \eqref{stima-v}, we deduce that
\begin{align}
\|\bm{v}(t,\cdot)\|_{\mathscr{X}_{2+\alpha}}\le Ce^{-\omega_0t}\|\bm{v}_0\|_{\mathscr{X}_{2+\alpha}},\qquad\;\,t\ge 0.
\label{stima-vv}
\end{align}

Next, we go back to function $u$, by means of formula \eqref{form-v-f}, which shows that
$u(t,\xi)=v_1(t,\xi-\eta(t)+\eta^*)+K(\xi)$ if $t\ge 0$ and $\xi<\eta(t)$ and $u(t,\xi)=v_2(t,\xi-\eta(t)+\eta^*)+K(\xi)$ if $t\ge 0$ and $\xi>\eta(t)$.
Moreover, $\eta(t)-\eta^*=f(t)=\Psi(w_1(t,\eta^*))$ for every $t\ge 0$, again due to formula \eqref{invert}.
It turns out that $\eta\in C^{1+\alpha/2}_b([0,+\infty))$ and
\begin{equation}
|\eta(t)-\eta^*|\le C_4e^{-\omega_0t}|v_{0,1}(\eta^*)|=Ce^{-\omega_0t}|u_0(\eta(0))-K(\eta^*)|,\qquad\;\,t\ge 0.
\label{stima-eta}
\end{equation}
We have so proved estimate \eqref{main-1}.
Moreover, since $q_L(\xi)=q_L(\xi-\eta(t)+\eta^*)q_L(\eta^*-\eta(t))$ and a similar formula holds true with $q_L$ being replaced by $q_H$, we can infer that
\begin{eqnarray*}
\frac{u(t,\xi)}{q_L(\xi)}-\frac{K(\xi)}{q_L(\xi)}
=\frac{v_1(t,\xi-\eta(t)+\eta^*)}{q_L(\xi)}
=\frac{v_1(t,\xi-\eta(t)+\eta^*)}{q_H(\xi-\eta(t)+\eta^*)}q_L(\eta(t)-\eta^*),
\end{eqnarray*}
for every $t\ge 0$ and $\xi\le\eta(t)$, and
\begin{eqnarray*}
\frac{u(t,\xi)}{q_H(\xi)}-\frac{K(\xi)}{q_H(\xi)}
=\frac{v_1(t,\xi-\eta(t)+\eta^*)}{q_H(\xi-\eta(t)+\eta^*)}q_H(\eta(t)-\eta^*),
\end{eqnarray*}
functions $\frac{u}{q_L}$ and $\frac{u}{q_H}$ are smooth in space and time and 
\begin{align*}
&\|u^{\sharp}(t,\cdot)-K^{\sharp}\|_{C^{2+\alpha}_b((-\infty,\eta(t)])}\le \|v_1^{\sharp}(t,\cdot)\|_{C^{2+\alpha}_b((-\infty,\eta^*])}
q_L(\eta(t)-\eta^*)\le Ce^{-\omega_0t}\|\bm{v}_0\|_{{\mathscr X}_{2+\alpha}}\\
&\|u^{\sharp}(t,\cdot)-K^{\sharp}\|_{C^{2+\alpha}_b([\eta(t),+\infty)))}\le \|v_2^{\sharp}(t,\cdot)\|_{C^{2+\alpha}_b([\eta^*,+\infty))}
q_H(\eta(t)-\eta^*)\le Ce^{-\omega_0t}\|\bm{v}_0\|_{{\mathscr X}_{2+\alpha}}
\end{align*}
for every $t\ge 0$, where we took into account \eqref{stima-vv} and \eqref{stima-eta}, which shows that the functions $q_L(\eta-\eta^*)$ and $q_H(\eta-\eta^*)$ are bounded in $[0,+\infty)$. Since
\begin{eqnarray*}
\|\bm{v}_0\|_{\mathscr{X}_{2+\alpha}}\le C\big (\|u_0^{\sharp}-K^{\sharp}\|_{C^{2+\alpha}_b((-\infty,\eta^*])}
+\|u_0^{\sharp}-K^{\sharp}\|_{C^{2+\alpha}_b([\eta^*,+\infty)}\big ),
\end{eqnarray*}
where $u_0=u(0,\cdot)$, from the previous two estimate we can infer that
\begin{align*}
&\|u^{\sharp}(t,\cdot)-K^{\sharp}\|_{C^{2+\alpha}_b((-\infty,\eta(t)])}\le Ce^{-\omega_0t}\big (\|u_0^{\sharp}-K^{\sharp}\|_{C^{2+\alpha}_b((-\infty,\eta^*])}
+\|u_0^{\sharp}-K^{\sharp}\|_{C^{2+\alpha}_b([\eta^*,+\infty)}\big )\\
&\|u^{\sharp}(t,\cdot)-K^{\sharp}\|_{C^{2+\alpha}_b([\eta(t),+\infty)))}\le Ce^{-\omega_0t}\big (\|u_0^{\sharp}-K^{\sharp}\|_{C^{2+\alpha}_b((-\infty,\eta^*])}
+\|u_0^{\sharp}-K^{\sharp}\|_{C^{2+\alpha}_b([\eta^*,+\infty)}\big ).
\end{align*}

Finally, we go back to problem \eqref{equ_phi}, \eqref{interface_phi}. Using condition \eqref{u-phi}, it is immediate to check that
$s(t)=\eta(t)-ct$ fot every $t\ge 0$ and
$\phi(t,x)=v(t,x+ct)$ for every $t\ge 0$ and $x\in\R$. Therefore,
\begin{eqnarray*}
\frac{\phi(t,x)}{q_L(x)}-\frac{K(x)}{q_L(x)}=\frac{u(t,x+ct)}{q_L(x)}-\frac{K(x+ct)}{q_L(x+xt)}
=\bigg (\frac{u(t,x+ct)}{q_L(t,x+ct)}-\frac{K(t,x+ct)}{q_L(t,x+ct)}\bigg )q_L(ct)
\end{eqnarray*}
for every $t\ge 0$ and $x\le s(t)$ and
\begin{eqnarray*}
\frac{\phi(t,x)}{q_H(x)}-\frac{K(x)}{q_H(x)}=\frac{u(t,x+ct)}{q_H(x)}-\frac{K(x+ct)}{q_H(x+xt)}
=\bigg (\frac{u(t,x+ct)}{q_H(t,x+ct)}-\frac{K(t,x+ct)}{q_H(t,x+ct)}\bigg )q_H(ct)
\end{eqnarray*}
for every $t\ge 0$ and $x\le s(t)$.
It thus follows that 
\begin{align*}
\|\phi^{\sharp}(t,\cdot)-e^{-rt}K^{\sharp}\|_{C^{2+\alpha}_b((-\infty,s(t)])}\le &\|u^{\sharp}(t,\cdot)-K^{\sharp}\|_{C^{2+\alpha}_b((-\infty,\eta(t)]}q_L(ct)\\
\le &
Ce^{-\sigma_-t}
\big (\|\phi_0^{\sharp}-K^{\sharp}\|_{C^{2+\alpha}_b((-\infty,s(0)])}
+\|\phi_0^{\sharp}-K^{\sharp}\|_{C^{2+\alpha}_b([s(0),+\infty)}\bigg )
\end{align*}
and
\begin{align*}
\|\phi^{\sharp}(t,\cdot)-e^{-rt}K^{\sharp}\|_{C^{2+\alpha}_b([s(t),+\infty))}\le &\|u^{\sharp}(t,\cdot)-K^{\sharp}\|_{C^{2+\alpha}_b([s(t),+\infty))}q_L(ct)\\
\le &Ce^{-\sigma_+t}
\big (\|\phi_0^{\sharp}-K^{\sharp}\|_{C^{2+\alpha}_b((-\infty,s(0)])}
+\|\phi_0^{\sharp}-K^{\sharp}\|_{C^{2+\alpha}_b([s(0),+\infty)}\bigg )
\end{align*}
for every $t\ge 0$, where $\sigma_-=-\frac{1-C_L}{2}c+\omega_0+r$ and $\sigma_+=
\frac{1-c_H}{2}c+\omega_0+r$.

Since $c=r-\delta$ we can write
\begin{eqnarray*}
\sigma_-=r\frac{1-c_L}{2}+\omega_0+\frac{1-C_L}{2}\delta,
\end{eqnarray*}
\begin{eqnarray*}
\sigma_+=\frac{1-c_H}{2}c+\omega_0+r=\frac{3-c_H}{2}r+\omega+\frac{c_H-1}{2}\delta
\end{eqnarray*}
and the right-hand sides of the previous formulas are positive since we are assuming that $c_L\in (0,1)$ and $c_H\in (1,3]$ (see  \eqref{admissible_bis}). The proof is complete.
\end{proof}

\section*{acknowlegments}
C.-M. B. greatly acknowledges the School of Mathematical Sciences of Tongji University for the warm hospitality during his visiting position from 2019 to 2021. He also expresses his gratitude to the Department of Mathematics and Computer Sciences of the University of Parma. C.-M. B. and J. L. acknowledge partial support from National Natural Science Foundation of China (No. 12071349). Y. D. acknowledges partial support from National Natural Science Foundation of China (No. 12071333 \& No. 12101458). L. Lorenzi is member of G.N.A.M.P.A. of the Italian Istituto Nazionale di Alta Matematica (INdAM).

\end{document}